\documentclass[reqno]{amsart}
\usepackage{amssymb}
\usepackage{hyperref}
\usepackage{graphicx,color}
\usepackage[section]{algorithm}
\usepackage{algorithmic}
\usepackage{enumerate}
\usepackage{mathrsfs,enumerate}
\usepackage{multirow,array}
\usepackage{subfigure}
\usepackage{pstricks,pst-node,pst-tree, pst-plot, pst-eps}
\usepackage{xspace}
\usepackage{marginnote}
\usepackage{upgreek}
\usepackage{bm}
\usepackage{comment}
\usepackage{tikz}
\usetikzlibrary{matrix}

\chardef\atsign='100

\newcommand{\cH}{H_{\#}}
\newcommand{\Forall}{\quad\text{for all }}
\newcommand{\HH}{H}
\newcommand{\HHs}{H_{\#}}
\newcommand{\dH}{\dot H}

\newcommand{\bad}{\bm a_d}

\newcommand{\vh}{\mathbb V(\mathcal T)}

\newcommand{\as}{\alpha^*}
\newcommand{\cQ}{\mathcal Q}
\newcommand{\bx}{\mathbf x}

\newcommand{\bP}{\mathbf P}
\newcommand{\Ps}{P_\#}

\theoremstyle{plain}% default
\newtheorem{theorem}{Theorem}[section]
\newtheorem{lemma}[theorem]{Lemma} 
\newtheorem{proposition}[theorem]{Proposition}
\newtheorem{corollary}[theorem]{Corollary}

\newtheorem{remark}{Remark}[section]

\graphicspath{{figures/}}
\begin{document}

\title[Fractional powers of the Laplace-Beltrami Operator]
{Approximation of the spectral fractional powers of the Laplace-Beltrami Operator}

\author[A.~Bonito]{Andrea Bonito}
\address[A.~Bonito]{Department of Mathematics, Texas A\&M University, College Station, TX 77843, USA}
\email{bonito@math.tamu.edu}
\thanks{A.B. is partially supported by the NSF Grants DMS-1817691 and DMS-2110811.}

\author[W.~Lei]{Wenyu Lei}
\address[W.~Lei]{SISSA - International School for Advanced Studies, Via Bonomea 265, 34136 Trieste, Italy}
\email{wenyu.lei@sissa.it}

\date{\today}

\begin{abstract}
	We consider numerical approximations of spectral fractional Laplace-Beltrami problems on closed surfaces. The proposed numerical algorithms rely on their Balakrishnan integral representation and consist of a sinc quadrature coupled with standard finite element methods for parametric surfaces. Possibly up to a log term, optimal rates of convergence are observed and derived analytically when the discrepancies between the exact solution and its numerical approximations are measured in $L^2$ and $H^1$. The performances of the algorithms are illustrated in different settings including the approximation of Gaussian fields on surfaces.
\end{abstract}



\maketitle

%%%%%%%%%%%%%%%%%%%%%%%%%%%%%%%%%%%%%%%%%%%%%%%%%%
\section{Introduction}\label{s:introduction}

In this paper, we consider finite element approximations of the spectral fractional Laplace-Beltrami problem
\begin{equation}\label{e:LB}
	(-\Delta_\gamma)^s \widetilde u = \widetilde f \qquad \textrm{on }\gamma,
\end{equation}
where $\gamma$ is a closed, compact and $C^3$ hyper-surface in $\mathbb R^n$ ($n=2,3$) and where the data $f$ is square integrable in $\gamma$ while satisfying the compatibility condition $\int_\gamma \widetilde f  = 0$.
Such system finds applications for instance in the study of Gaussian random fields on surfaces \cite{lindgren2011explicit, lang2015isotropic,bolin2020numerical}.

As we shall see, the spectral negative fractional powers of the Laplace-Beltrami operator $L:=-\Delta_\gamma$ are defined using the Balakrishnan formula \cite{Balakrishnan60}
\begin{equation}\label{e:frac-int}
	L^{-s}\widetilde f:= \frac{\sin(\pi s)}{\pi} \int_0^\infty \mu^{-s}(\mu I+L)^{-1} \widetilde f\, d\mu, \qquad s\in (0,1).
\end{equation}
This representation is equivalent to the spectral representation more often used as detailed in Section~\ref{ss:dotted}.

On Euclidean domains, a two-step numerical approximation strategy based on \eqref{e:frac-int} is proposed in \cite{BP13,BP15} for the standard Laplacian.
After a change of variable, an exponentially converging sinc quadrature \cite{BLP17}
\begin{equation}\label{e:sinc-approx}
	\cQ^{-s}_k(L) \widetilde f = \frac{k\sin(\pi s)}{\pi} \sum_{\ell=-M}^N e^{(1-s) y_\ell} (e^{y_\ell}I+  L)^{-1} \widetilde f,
\end{equation}
is advocated for the approximation of the outer integral in \eqref{e:frac-int}.
Here $k>0$ is the quadrature spacing, $y_\ell=\ell k$,  and $M, N \in O(k^{-2})$ are positive integers.
In \cite{BP13,BP15}, each of the sub-problems $(e^{y_\ell}I+L)^{-1}\widetilde f$ are approximated independently using a standard continuous piecewise linear finite element method. We refer to \cite{banjai2020exponential} for extensions to $hp$-discretization methods. For completeness, we also mention the reviews \cite{bonito2018numerical,lischke2018fractional} describing alternate algorithms for fractional powers of elliptic operators on Euclidean domains.

Regarding the Laplace-Beltrami operator on closed surface, several algorithms for its approximation are available as well.
Parametric finite element methods \cite{dziuk88} and trace finite element methods \cite{ORG09,reusken14,DER14} rely on polygonal approximations $\Gamma$ of the exact surface $\gamma$. Instead, other numerical methods like the narrow band methods \cite{DDEH09,DER14} are defined on a neighborhood of $\gamma$.
We refer to \cite{bonito2020finite,Dziuk13} for reviews of these methods and others.
All the above mentioned numerical approaches have in common an inherent geometric errors (or consistency errors) resulting from the approximation of $\gamma$, thereby leading to the uncharted territory (in the context of approximation of fractional operators) of non-conforming methods.

The main contribution of this work is to develop a new analysis to incorporate the approximation by parametric finite element methods within the framework developed in \cite{BP15}. The approximate polygonal surface on which the finite element method is defined is related to the exact surface via a $C^2$ orthogonal projection $\bP: \Gamma \rightarrow \gamma$ given by $\bP(\bx) =  \bx - d(\bx) \nabla d(\bx)$, where $d$ is the signed distance function to $\gamma$.
In \cite{dziuk88,demlow2007adaptive} this typical choice of lift is put forward in view of its $O(h^2)$ approximation of the geometry at the expense of requiring surfaces $\gamma$ of class $C^3$. Here $h$ stands for typical diameter of the faces constituting the approximation $\Gamma$.
In opposition, favoring practical implementations and applicability to rougher surfaces, a generic continuous piecewise $C^2$ lift is considered in \cite{mekchay2011afem,bonito2013afem,bonito2016high} leading to a reduced $O(h)$ approximation of the geometry.
The latter is sufficient to control with optimal order the $H^1$, but not $L^2$, discrepancy between the solution to the standard Laplace-Beltrami operator and its parametric linear finite element approximations.
However, the analysis of algorithms based on the integral representation \eqref{e:frac-int} requires space discretization methods to perform simultaneously well in $H^1$ and $L^2$, thereby exacerbating even more the intricate relation between surface approximations and finite element approximations on approximate surfaces.
For completeness, we mention that recently the two approaches were reconciled in \cite{bonito2019posteriori} for $C^3$ surface, where a generic lift is used to define the algorithm leaving the use of the signed distance function only as theoretical tool.
Later in \cite{bonito2020finite}, this technology has been extended to $C^2$ surfaces. Whether these new developments can be applied to the approximations of fractional powers of the Laplace-Beltrami operator remains open.

In Section~\ref{s:prelim}, we introduce preliminary notations and justify the integral representation \eqref{e:frac-int}.
In Section~\ref{s:simulation}, we describe the proposed algorithms and present numerical simulations indicating that the expected rate of convergence of the resulting approximations matches the Euclidean case \cite{BP13, BP15} when using the signed distance function or generic lifts.
Following  \cite{bolin2020numerical}, we also illustrate how the proposed algorithms can be used for efficient approximation of  Gaussian random fields on surfaces. 
The convergence rates obtained in Section~\ref{s:simulation} are investigated in Section~\ref{s:analysis}, where it is shown that the method delivers (up to a logarithmic term) error estimate of order $O(e^{-c/k}) + O(h^{\min(2,\delta+s)})$ in $L^2$  whenever the data $\widetilde f$ as fractional regularity in $L^2$ of order $2\delta$ and of order  $O(e^{-c/k}) + O(h^{\min(2,\delta+s)-1})$ in $H^1$ under the additional restriction $\delta+s>\frac 1 2$.
The main results are Theorem~\ref{t:rate} (space discretization) and Corollary~\ref{c:rate} (total error).

In what follows, we write $a \lesssim b$ to mean $a \leq C b$, with a constant $C$ that does not depend on $a$, $b$, or  the discretization parameters. Finally, $a \sim b$ indicates $a\lesssim b$ and $b\lesssim a$. We also denote $\|F\|_{X\to Y}$ the induced norm of the linear functional $F$ mapping from $X$ to $Y$. We reserve the $\widetilde .$ notation for quantities defined on $\gamma$ and discrete functions in space are denoted with capitals.

%%%%%%%%%%%%%%%%%%%%%%%%%%%%%%%%%%%%%%%%%%%%%%%%%%%%%%%%
\section{Fractional Powers of the Laplace-Beltrami Operator}\label{s:prelim}

The aim of this section is to provide a definition of the fractional powers of the Laplace-Beltrami operator.
We start by introducing interpolation spaces.

\subsection{Interpolation Spaces}

Let $L^2(\mathcal S)$ be the space of square integrable functions on a Lipschitz hyper-surface $\mathcal S$ and $(.,.)_{\mathcal S}$ be the associated inner product.
We use $H^r(\mathcal S)$ to denote the Hilbert space consisting of $L^2$-functions with all weak derivatives of total order $\leq r$ in $L^2$. To account for the compatibility conditions satisfied by $\widetilde f$ in \eqref{e:LB-strong}, we also consider the subspaces of vanishing mean value functions
\begin{align*}
	L^2_\#(\mathcal S) & := H^0(\mathcal S) := \left\{ v\in L^2(\mathcal S) : ( v,1)_{\mathcal S}=0\right\}, \\
	\cH^i(\mathcal S)  & := H^i(\mathcal S) \cap L^2_\#(\mathcal S), \quad i=1,2,3,....
\end{align*}
For $0\leq s \leq 2$, we denote by  $H^s(\mathcal S):= [L^2(\mathcal S),H^2(\mathcal S)]_{s/2,2}$ the fractional spaces, where $[X,Y]_{r,2}$ are the intermediate spaces between $X$ and $Y$ obtained using the real interpolation method.
Similarly, we set $H^s_\#(\mathcal S):= [L^2_\#(\mathcal S),H^2_\#(\mathcal S)]_{s/2,2}$ and point out that they are equivalent to $H^s(\mathcal S) \cap L^2_\#(\mathcal S)$ with a norm equivalence constant depending on $\mathcal S$ according to Lemma 2.1 and A1 in \cite{guermond2009lbb}.
The dual spaces of $H^s(\mathcal S)$ (resp. $H^s_\#(\mathcal S)$) are denoted  $H^{-s}(\mathcal S)$ (resp. $H^{-s}_\#(\mathcal S)$).
We use $\langle .,. \rangle_{A}$ to denote the duality pairing between $A$ and its dual.
Furthermore, we identify $L^2(\mathcal S)$ with its dual leading to the embeddings
$$
	H^s(\mathcal S) \subset L^2(\mathcal S) \subset H^{-s}(\mathcal S), \qquad H^s_\#(\mathcal S) \subset L^2_\#(\mathcal S) \subset \cH^{-s}(\mathcal S), \quad s \geq 0,
$$
which are frequently used in the following without further mentioning it.

\subsection{The Laplace-Beltrami Problem}\label{ss:LB}
We assume that $\gamma$ is a $C^3$ closed hyper-surface of $\mathbb R^n$, $n=2,3$, and that we have access to its signed distance function $d:\mathbb R^n \rightarrow \mathbb R$.  The latter induces an orthogonal projection
\begin{equation}\label{e:orthogonal}
	\mathbf P(\bx) = \bx - d(\bx) \nabla d(\bx)
\end{equation}
to $\gamma$ in a tubular neighborhood $\mathcal N$ of $\gamma$ of fixed diameter depending on the curvature of $\gamma$; refer for e.g. to \cite{bonito2020finite} and the reference therein for more details.
This allows to define the tangential operators using extensions and standard differential operators in $\mathbb R^n$. For example, the surface gradient is given by
$$
	\nabla_\gamma \widetilde v = (I-\nabla d \otimes \nabla d) \nabla v|_{\gamma}, \qquad \Delta_\gamma := \nabla_\gamma \cdot \nabla_\gamma,
$$
where $v$ is the extension of $\widetilde v$ to $\mathcal N$ defined as $v(\bx):=(P\widetilde v)(\bx):=\widetilde v(\bP(\bx))$, $\bx \in \mathcal N$.

Given $\widetilde f \in H^{-1}_\#(\gamma)$, we consider the weak formulation
\begin{equation}\label{e:dirichlet}
	a_\gamma(\widetilde w,\widetilde v):=  \int_\gamma \nabla_\gamma \widetilde w \cdot \nabla_\gamma \widetilde  v =  \langle \widetilde f ,\widetilde v \rangle_{H^1_\#(\gamma)}, \Forall \widetilde v\in \cH^1(\gamma),
\end{equation}
of the Laplace-Beltrami problem
\begin{equation}\label{e:LB-strong}
	-\Delta_\gamma \widetilde w = \widetilde f.
\end{equation}
The bilinear form $a_\gamma: H^1(\gamma) \times H^1(\gamma)$ is continuous and thanks to the  Poincar\'e-Friederich inequality
\begin{equation}\label{e:poincare}
	\| \tilde v \|_{L^2(\gamma)} \lesssim \| \nabla_\gamma  \tilde v\|_{L^2(\gamma)}, \qquad \forall  \tilde v \in H^1_\#(\gamma),
\end{equation}
see e.g. \cite{bonito2020finite}, it is coercive on $H^1_\#(\gamma)$.
Thus, the Lax-Milgram theory guarantees that there exists a unique $\widetilde w \in \cH^1(\gamma)$ satisfying the variational problem \eqref{e:dirichlet}.
We denote by $T: L^2_\#(\gamma)\to \cH^1(\gamma)$ the solution operator $T\widetilde f:=\widetilde w$ and by $L := (-\Delta_\gamma):=T^{-1} $ its inverse with domain $D(L):=\text{Range}(T)$.

\subsection{Fractional Powers and Dotted Spaces}\label{ss:dotted}
By construction, the unbounded operator $L$ is regularly accretive.
The negative fractional powers $L^{-s}=(-\Delta_\gamma)^{-s}$, $s\in (0,1)$,  can be defined via the general framework of regularly accretive operators \cite{Kato61}: for $\widetilde f \in L^2_\#(\gamma)$, we set
\begin{equation}\label{e:balakrishnan}
	(-\Delta_\gamma)^{-s}\widetilde  f := \frac{\sin(\pi s)}{\pi} \int_0^\infty \mu^{-s}(\mu I-\Delta_\gamma)^{-1} \widetilde f\, d\mu .
\end{equation}
The above representation is referred to as the Balakrishnan formula \cite{Balakrishnan60}.
For positive powers, we denote by $D(L^s):= \{ v \in L^2_\#(\mathcal \gamma) \ : \  L^{s-1} v \in D(L)\}$ the domain of the fractional operator $L^s$ and $(-\Delta_\gamma)^s \widetilde f := L(L^{s-1}\widetilde f)$. Our goal is to approximate
\begin{equation}\label{e:sol_u}
	\widetilde u := (-\Delta_\gamma)^{-s} \widetilde f
\end{equation}
for $\widetilde f\in L^2_\#(\mathcal \gamma)$ and $s\in(0,1)$.

Although the Balakrishnan formula \eqref{e:balakrishnan} is the basis of the proposed numerical scheme, we now show it reduces to a more standard spectral decomposition.
The compact embedding $\cH^1(\gamma) \subset  L^2(\gamma)$ guarantees that there exists an $L^2(\gamma)$-orthonormal basis of eigenfunctions $\{\widetilde \psi_j\}_{j=1}^\infty \subset \cH^1(\gamma)$ of $T$ with non-increasing positive eigenvalues $\{\widetilde \mu_j\}_{j=1}^\infty$. Note that $\{\widetilde \psi_j\}_{j=1}^\infty$ is also an $L^2(\gamma)$-orthonormal basis of eigenfunctions of $L$ with corresponding eigenvalues $\widetilde \lambda_j=1/\widetilde \mu_j$. The spectral decompositions of $L$ offers a representation of the fractional powers for $-1<s<1$, namely
\begin{equation}\label{e:L_expension}
	L^s\widetilde v=\sum_{j=1}^\infty \widetilde \lambda_j^s \langle \widetilde v,\widetilde \psi_j \rangle_{H^1_\#(\gamma)}\widetilde  \psi_j.
\end{equation}
Furthermore, the domains $D(L^s)$ of the fractional operators satisfy $D(L^s)=\dH^{2s}_\#(\gamma)$, where for $r \ge -1$, $\dH^r_\#(\gamma)$ is the set of functions in $H^{-1}_\#(\gamma)$ such that the norm
\begin{equation}\label{e:inter-def}
	\|\widetilde v\|_{\dH^r(\gamma)}:=\bigg(\sum_{j=1}^\infty \widetilde \lambda_j^r| \langle \widetilde v,\widetilde \psi_j \rangle_{H^1_\#(\gamma)}|^2\bigg)^{1/2}
\end{equation}
is finite. Refer for instance to \cite[Theorem 2.1]{BP15} for more details on this equivalence.
In particular, we deduce that
\begin{equation}\label{e:regularity-dot}
	(-\Delta_\gamma) \widetilde f \in \dH^{2s+r}(\gamma) \qquad \textrm{for }\widetilde f \in \dH^r(\gamma), \qquad r \ge -1.
\end{equation}

Furthermore, The spaces $\HHs^r(\gamma)$ and $\dH^r_\#(\gamma)$ coincide for $r\in[-1,2]$ with equivalent norms
\begin{equation}\label{e:norm_equi_dotted}
	\| . \|_{H^r(\gamma)} \sim \| . \|_{\dH^r(\gamma)} = \| L^{r/2} . \|_{L^2(\gamma)}.
\end{equation}
We refer to \cite{guermond2009lbb} for the cases $r \in[-1,1]$ and to \cite{BP15} for the extension to $r\in[1,2]$ pointing out that the latter hinges on the fact that
$$
	T\textrm{ is an isomorphism from }L^2_\#(\gamma)\textrm{ to }\HHs^{2}(\gamma)
$$
when $\gamma$ is of class $C^2$, see Theorem 3.3 in \cite{Dziuk13} or Lemma 3 in \cite{bonito2020finite}.

We end this section with some estimates involving the Laplace-Beltrami operator. We start by noting that the Poincar\'e-Friederich inequality \eqref{e:poincare} implies that
\begin{equation}\label{i:bound1}
	\|(\mu I-\Delta_\gamma)^{-1} \widetilde  g\|_{L^2(\gamma)}\lesssim \min(1,\mu^{-1}) \|\widetilde g\|_{L^2(\gamma)},
	\quad\text{for }\mu>0 \text{ and }\widetilde g\in L^2_\#(\gamma).
\end{equation}
One direct consequence of the norm equivalence \eqref{e:norm_equi_dotted} is that \eqref{e:regularity-dot} implies
\begin{equation}\label{e:regularity}
	(-\Delta_\gamma)^{-s} \widetilde f \in H_\#^{2s+r}(\gamma) \qquad \textrm{for }\widetilde f \in H_\#^r(\gamma)
\end{equation}
provided $r \geq -1$ and $2s+r\leq 2$. Furthermore there holds
\begin{equation}\label{e:regularity-estim}
	\| (-\Delta_\gamma)^{-s} \widetilde f\|_{H^{2s+r}(\gamma)} \lesssim \| \widetilde f\|_{H^r(\gamma)}.
\end{equation}
Another direct consequence of the norm equivalence \eqref{e:norm_equi_dotted}, is the following finer version of estimate \eqref{i:bound1}: Let $r\in[0,1]$, $t\in[0,2]$ satisfying $r+t\in [0,2]$. Given $\widetilde g\in \HHs^t(\gamma)$ and $\mu\in(0,\infty)$, there holds
\begin{equation}\label{e:bound}
	\|(-\Delta_\gamma)(\mu I -\Delta_\gamma)^{-1} \widetilde g\|_{\HH^{-r}(\gamma)}\lesssim \mu^{-(r+t)/2} \|\widetilde g\|_{\HH^t(\gamma)} .
\end{equation}
Refer to Lemma~6.5 in \cite{BP15} for a proof; see also Proposition~4.1 and Lemma~4.5 in \cite{BP13}.

%%%%%%%%%%%%%%%%%%%%%%%%%%%%%%%%%%%%%%%%%%%%%%%%%%%%%%%%
\section{Numerical schemes}\label{s:simulation}

\subsection{Outer Quadrature Formula}

We follow \cite{BLP17} and use a sinc numerical quadrature to approximate the integral with respect to the variable $\mu=e^y$ in \eqref{e:balakrishnan}:
\begin{equation}\label{e:approx_quad}
	(-\Delta_\gamma)^{-s} \widetilde f \approx \cQ^{-s}_k(L)\widetilde f := \frac{k\sin(\pi s)}{\pi} \sum_{\ell=-M}^N e^{(1-s) y_\ell} (e^{y_\ell}I-\Delta_\gamma)^{-1} \widetilde f,
\end{equation}
where $k>0$ is the quadrature spacing,
\begin{equation}\label{e:choose}
	N:=\bigg\lceil \frac {\pi^2}{4s k^2}\bigg\rceil \quad \hbox{ and }\quad
	M:=\bigg \lceil \frac {\pi^2}{4(1-s) k^2}\bigg \rceil,
\end{equation}
and $y_\ell := k \ell$, $\ell = -M,...,N$. This particular choice of $M$ and $N$ balances the three sources of quadrature errors: the approximation of the integral from $y_{-M}$ to $y_N$ and the contributions of the integrals from $-\infty$ to $y_{-M}$ as well as from $y_N$ to $\infty$ not accounted for in \eqref{e:approx_quad}, see \cite[Remark 3.1]{BLP17}.
We shall see in Section~\ref{ss:sinc} that this quadrature approximation is exponentially convergent in $k$.

In order to simplify the notations, we denote by  $\widetilde u^{\ell}:= \widetilde u^\ell (f)$ the unique function in $\cH^1(\gamma)$ satisfying  the variational formulation
$$
	e^{y_\ell}(\widetilde u^\ell,\widetilde v) + a_\gamma(\widetilde u^\ell,\widetilde v) = (\widetilde f,\widetilde v),\quad\forall \widetilde v\in \cH^1(\gamma).
$$
Note that for $\widetilde f \in L^2_\#(\gamma)$, we have $\int_\gamma \widetilde f = 0$ and so $\widetilde u^\ell \in \cH^1(\gamma)$ is also characterized as the unique function in $H^1(\gamma)$ satisfying
\begin{equation}\label{e:sub-problem}
	e^{y_\ell}(\widetilde u^\ell,\widetilde v) + a_\gamma(\widetilde u^\ell,\widetilde v) = (\widetilde f,\widetilde v),\quad\forall \widetilde v\in H^1(\gamma),
\end{equation}
which is more amenable to finite element approximations.
Regardless, using $\widetilde u^\ell$ in the quadrature approximation \eqref{e:approx_quad} we arrive at
\begin{equation}\label{e:quadrature}
	\widetilde u_k:=\cQ^{-s}_k(L)\widetilde f = \frac{k\sin(\pi s)}{\pi} \sum_{\ell=-M}^N e^{(1-s) y_\ell} \widetilde u^\ell.
\end{equation}

\subsection{Parametric Finite Element Method}\label{ss:parametric}

We now discuss the approximation of elliptic operator on hype-surface using the parametric finite element methods.
We assume that $\Gamma$ is a $(n-1)$-dimensional polyhedral surface lying in $\mathcal N$ so that $\mathbf P$ in \eqref{e:orthogonal} is a $C^2$ diffeomorphism from $\Gamma$ to $\gamma$.
The orthogonal projection \eqref{e:orthogonal} provided a higher order approximation of the geometry compared to generic lifts $\bP$, see e.g. \eqref{i:q-estimate}.
However, this is at the expense of requiring $C^3$ smoothness on the surface rather than $C^2$.
In \cite{bonito2020finite}, it is shown that it was possible to win both ways and design an algorithm for the Laplace-Beltrami problem using a generic lift while taking advantage of the orthogonality of \eqref{e:orthogonal} only in the analysis for $C^2$ surface. Whether this technology can be adapted to the current context remains an open problem not addressed in the current work (see Conclusions section) but point out that the numerical observations reported in Section~\ref{ss:numerical} indicate that the convergence of the algorithm is not affected when using generic lifts.
We thus restrict our consideration to algorithms defined using the signed distance function lift \eqref{e:orthogonal} and assume, to avoid additional technicalities, that the vertices of $\Gamma$ lie on $\gamma$ (see \cite{demlow09} for a discussion on how to relax such assumption).
Furthermore, from now on we assume that the geometry of $\gamma$ is sufficiently well approximated by $\Gamma$ so that $\Gamma \subset \mathcal N$ and thus the lift $\bP$ is well defined on $\Gamma$.

We denote by $\mathcal T$ the collection of the faces\footnote{We use the names of quantities corresponding to 2-dimensional surfaces even if $\gamma$ could be a curve.}, which are assumed to be all either triangles or quadrilaterals. When using triangular subdivisions, the reference element $\widehat \tau$ is the unit triangle and we set $\mathbb P$ to be the set of linear polynomials. Instead, when the subdivisions are made of quadrilaterals, the reference element $\widehat \tau$ is the unique square and we set $\mathbb P$ to be the set of bi-linear polynomials.
Associated to each face $\tau \in \mathcal T$, we denote by $F_\tau: \widehat \tau \rightarrow \tau \in [\mathbb P]^{n-1}$ the map from the reference element to the physical element.
We let $c_J:=c_J(\mathcal T)$ be such that
$$
	c_J^{-1} | \mathbf{w} | \leq  | D F_\tau \mathbf{w} | \leq c_J | \mathbf{w}|, \qquad \forall \mathbf{w} \in \mathbb R^{n-1}, \forall \tau \in \mathcal T.
$$
Furthermore, we denote by $h_\tau$, $\tau\in \mathcal T$, the diameter of $\tau$ and by $c_q:=c_q(\mathcal T)$ the quasi-uniformity constant
$$
	h:=\max_{T\in \mathcal T} h_T = c_q \min_{T\in \mathcal T} h_T.
$$

For polyhedral surfaces, the number of faces sharing the same vertex is not controlled by the quasi-uniformity constant $c_q$ (think of a surface zigzaging around a vertex as depicted in Fig. 1 in \cite{bonito2019posteriori}). Thus, we denote by $c_v:=c_v(\mathcal T)$ this valence, namely
\begin{equation}\label{e:uniform-vertices}
	c_v:=\max_{ \mathbf v \textrm{ vertex of } \mathcal T} \# \{ T \in \mathcal T \ : \ \mathbf v \textrm{ is a vertex of }T\}.
\end{equation}
The constants appearing in the analysis below may depend on $c_J$, $c_q$ and $c_v$ but not on $h$.
At this point, it is worth mentioning that there are algorithms constructing sequences of surface approximations with uniformly bounded $c_J$, $c_q$ and $c_v$.
Examples are subdivisions $\{\gamma_i:= \mathcal I_i \bP(\overline{\gamma}_i)\}_{i=1}^\infty$, where $\mathcal I_i$ is the standard Lagrange nodal interpolant associated with the subdivision $\overline{\gamma}_i$ and $\{ \overline{\gamma}_i \}_{i=1}^\infty$ is a sequence of uniform refinements of an initial polyhedral surface $\overline{\gamma}_0$; refer for instance to \cite{bonito2012convergence} for subdivisions obtained by uniform refinements and to  \cite{bonito2013afem,bonito2016high} for adaptively refined meshes.

%We shall define our discrete problem following \cite[Chapter 1]{bonito2020finite}. We first denote $d(x)$ the signed distance function of $\gamma$ and let $\mathcal{N}_\delta$ be a tubular neighborhood of $\gamma$ with a strip width $\delta$. We assume that $\delta$ is small enough to guarantee that the decomposition $a(\bx) = \bx-d(\bx)\nv(\bx)$ for $\bx\mathcal \in \mathcal N_\delta$ is unique. For any function $f$ defined on $\mathcal N_\delta$, denote the lifting operator $\bad f = f\circ a$. Let $\Gamma\subset \mathcal{N}_\delta$ with $h>0$ denoting the mesh size be a polyhedron in $2d$ or polytope in $3d$ made of simplices so that its vertices lie on $\gamma$. Denote the corresponding subdivision $\mathcal T_{\mathcal T}$ (i.e. the union of simplices) and assume that $\mathcal T_{\mathcal T}$ is shape regular and quasi-uniform (cf. \cite{EG13}). 

The finite element space associated with the subdivision $\mathcal T$ is denoted $\mathbb V(\mathcal T)$ and is given by
$$
	\mathbb V(\mathcal T):= \{ v \in H^1(\Gamma) \ : \ v|_\tau \circ F_\tau  \in \mathbb P, \qquad \forall \tau\in \mathcal T \}.
$$
Note that $\mathbb V(\mathcal T)$ is not restricted to vanishing mean value functions. Rather, it is tailored for the approximation of the sub-problems \eqref{e:sub-problem}. More precisely, we propose to approximate $\widetilde u^\ell$ in \eqref{e:sub-problem} by $U^\ell \in \mathbb V(\mathcal T)$ satisfying
\begin{equation}\label{e:discrete-parametric}
	e^{y_\ell}\int_{\Gamma} U^\ell V  +
	\int_{\Gamma} \nabla_{\Gamma} U^\ell \cdot\nabla_{\Gamma}V
	= \int_{\Gamma}  V P \widetilde f  \sigma,
	\quad\text{ for all } V\in \mathbb V(\mathcal T) .
\end{equation}
Recall that $P:L^2(\gamma)\to L^2(\Gamma)$ is defined with $P\widetilde f = \widetilde f\circ\bP$ and $\sigma:\Gamma \rightarrow \mathbb R$ is the ratio between the area element of $\gamma$ and $\Gamma$ associated with the parametrization $\mathbf P:\Gamma \rightarrow \gamma$.
We  will also use the notation $P_\#:= \sigma P: L^2(\gamma) \rightarrow L^2(\Gamma)$ and note  that
$\int_{\Gamma} P_\# \widetilde f  = \int_{\Gamma} P \widetilde f \sigma = \int_{\gamma} \widetilde f =0$ whenever $\widetilde f \in L^2_\#(\gamma)$. As a consequence, $U^\ell$ satisfies  $\int_{\Gamma} U^\ell = 0$ as well.
In view of \eqref{e:approx_quad}, the approximation $U_{k} \in \mathbb V(\mathcal T) \cap L^2_\#(\Gamma)$ to $u$ in \eqref{e:sol_u} is defined as
\begin{equation}\label{e:sol_ukh}
	U_{k}:= \cQ_{k}^{-s}(L_{\mathcal T})\Ps\widetilde f = \frac{k\sin(\pi s)}{\pi} \sum_{\ell=-M}^N e^{(1-s) y_\ell} U^\ell \in \mathbb V(\mathcal T) \cap L^2_\#(\Gamma).
\end{equation}

	\begin{remark}[Other approximate right hand sides in \eqref{e:discrete-parametric}]
		The term $P\widetilde f \sigma$ in the right hand side \eqref{e:discrete-parametric} can be substituted by any $O(h^2)$ approximation of $P\widetilde f$ with vanishing mean value on $\Gamma$; refer to  \cite{demlow09} for a discussion on the resulting consistency term. For example, one could consider $I_h(P\widetilde f) - \frac 1 {|\Gamma|} \int_\Gamma I_h(P\widetilde f)$ where $I_h$ stands for the Lagrange interpolant. Whether one can relax the $O(h^2)$ approximation to an $O(h)$ approximation like for the standard case $s=1$ as in \cite{bonito2020finite} remains an open question.
	\end{remark}
%\begin{remark}
%	The above discretization setting can be also applied to other elements like $n$-quadrilaterals. From what follows, we only consider simplices for analysis in Section~\ref{s:analysis}.
%\end{remark}

\subsection{Numerical Illustration}\label{ss:numerical}
\subsubsection{Convergence Tests}
For this numerical experiment, $\gamma$ is the unit sphere in the three dimensional space.
The data $\widetilde f \in \cH^{\frac 12 -\varepsilon}$, $\varepsilon>0$,  is the step function
\begin{equation}\label{e:right-hand-side}
	\widetilde f(x_1, x_2, x_3) = \left\{
	\begin{aligned}
		 & 1,  & \quad \text{ if } x_3 \ge 0, \\
		 & -1, & \quad \text{ if } x_3 < 0 .  \\
	\end{aligned}
	\right .
\end{equation}
The exact solution $\widetilde u$ is represented using an eigenfunction expansion in the spherical coordinate system $(\theta,\phi)\in [0,\pi]\times[0,2\pi]$, with $\theta$ and $\phi$ indicating the latitudinal and longitudinal directions, respectively. That is
\[
	\widetilde u(\bx) = \widetilde u(\theta,\phi)
	= \sum_{j=1}^{\infty}\sum_{m=-j}^j \lambda_j^{-s} (\widetilde f,\widetilde Y_{mj}(\theta,\phi)) \widetilde Y_{mj}(\theta,\phi),
\]
where $\{\widetilde Y_{mj}\}$ for $j=0,1,\ldots$ and $m=-j,\ldots,j$ is the sequence of spherical harmonic functions. Since $\widetilde f$ is  independent of $\phi$, $(\widetilde f,\widetilde Y_{mj})=0$ for $m \not = 0$. Whence, we obtain
\[
	\widetilde u(\bx) = \widetilde u(\theta,\phi)
	= \sum_{j=1}^{\infty} \lambda_j^{-s} (\widetilde f,\widetilde \zeta_j(\theta,\phi)) \widetilde \zeta_j(\theta,\phi) ,
\]
where $\widetilde \zeta_j(\theta,\phi) = \widetilde Y_{0j}(\theta,\phi) = \sqrt{(2j+1)/(4\pi)}\, \widetilde p_j(\cos(\theta))$ and $\{\widetilde p_j\}_{j\geq 1}$ are the Legendre polynomials.
The first $10,000$ modes are retained for the evaluation of $\widetilde u(\bx)$ so that the overall error is not affected by the truncation.

The discrete surfaces $\Gamma$ are obtained by uniform refinements of the coarse subdivision depicted in Figure~\ref{f:uh-parametric}. The algorithm is implemented using the \texttt{deal.II} finite element library \cite{arndt2020deal}\footnote{see also a tutorial problem at https://www.dealii.org/current/doxygen/deal.II/step\_38.html.}. In terms of the matrix computation for the formula \eqref{e:sol_ukh}, we note that we invert the matrix system on the left hand side of \eqref{e:discrete-parametric} using the direct solver from \texttt{UMFPACK}\footnote{https://people.engr.tamu.edu/davis/suitesparse.html} but multigrid solvers \cite{kornhuber2008multigrid,bonito2012convergence} could be used for larger problems.
The sinc quadrature parameter $k$ is set to $k=0.15$ so that the space discretization error dominates the total error.

\begin{figure}[hbt!]
	\begin{center}
		\includegraphics[scale=0.15]{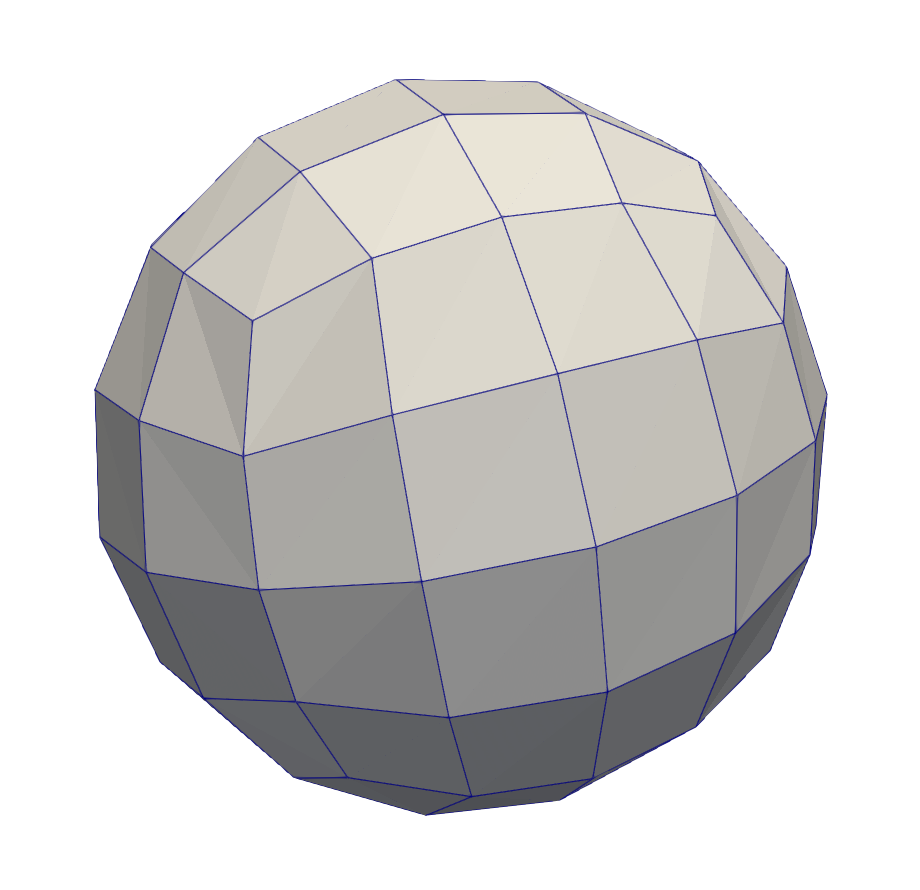}
	\end{center}
	\caption{A coarse quadrilateral subdivision of the unit sphere.}
	\label{f:uh-parametric}
\end{figure}

On Euclidean domains, we expect (Theorem 6.2 in \cite{BP15}) that for $\widetilde f\in \cH^{1/2-\epsilon}(\gamma)$ we have $\|P\widetilde u-U_k\|_{L^2(\Gamma)} \sim h^{\max(2,\frac12+2s)}\sim \#\text{DoFs}^{-\min(1,\frac14+s)}$ and $\|P\widetilde u- U_k\|_{H^1(\Gamma)} \sim h^{\max(1,2s-\frac12)}\sim \#\text{DoFs}^{-\min(\frac12,s-\frac14)}$. This matches the convergence rates observed for the Laplace-Beltrami operator as reported in Figure~\ref{f:convergence-parametric}.
The analysis below (see Theorem~\ref{t:rate}) provides a rigorous justification.
To illustrate the influence of the fractional power $s$ on the exact solution and its approximations, we also report  in Figure~\ref{f:extract} the approximate solution $U_k$ for $s=0.3$ and the values of $U_k$ for different $s$ along a geodesic from the south pole to the north pole.

\begin{figure}[hbt!]
	\begin{center}
		\begin{tabular}{cc}
			\includegraphics[scale=0.5]{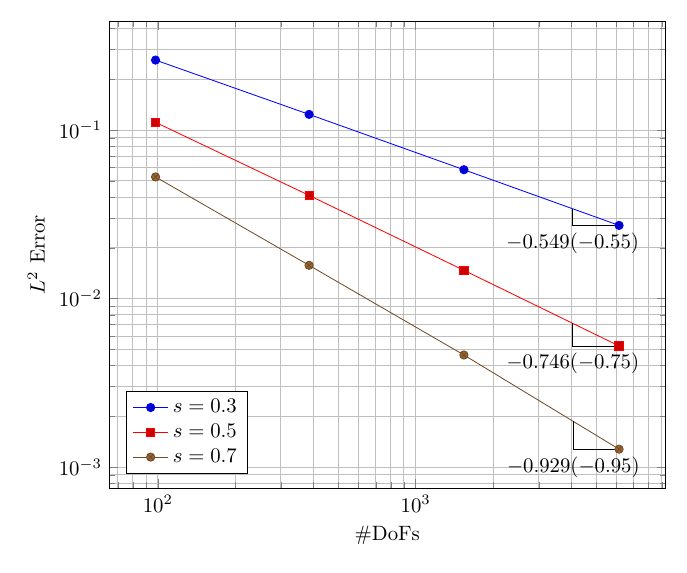} & \includegraphics[scale=0.5]{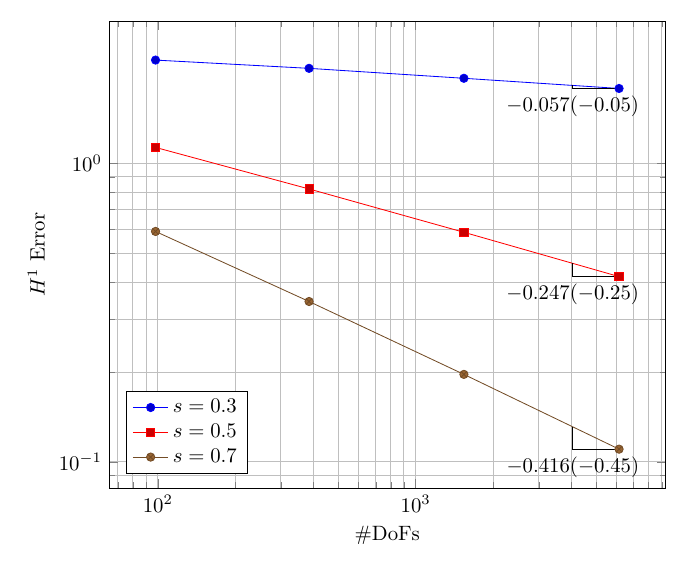}
		\end{tabular}
	\end{center}
	\caption{The errors $\|P\widetilde u-U_k\|_{L^2(\Gamma)}$ (left) and $\| P\widetilde u-U_k\|_{H^1(\Gamma)}$ (right)  against the number of degree of freedoms for $s=0.3$, $0.5$ and $0.7$. For each error plot, the slope of the last segment is reported along with the convergent rate (in parenthesis) guaranteed by Theorem~\ref{t:rate} below.}
	\label{f:convergence-parametric}
\end{figure}
\begin{figure}[hbt!]
	\begin{center}
		\begin{tabular}{cc}
			\includegraphics[height=6cm]{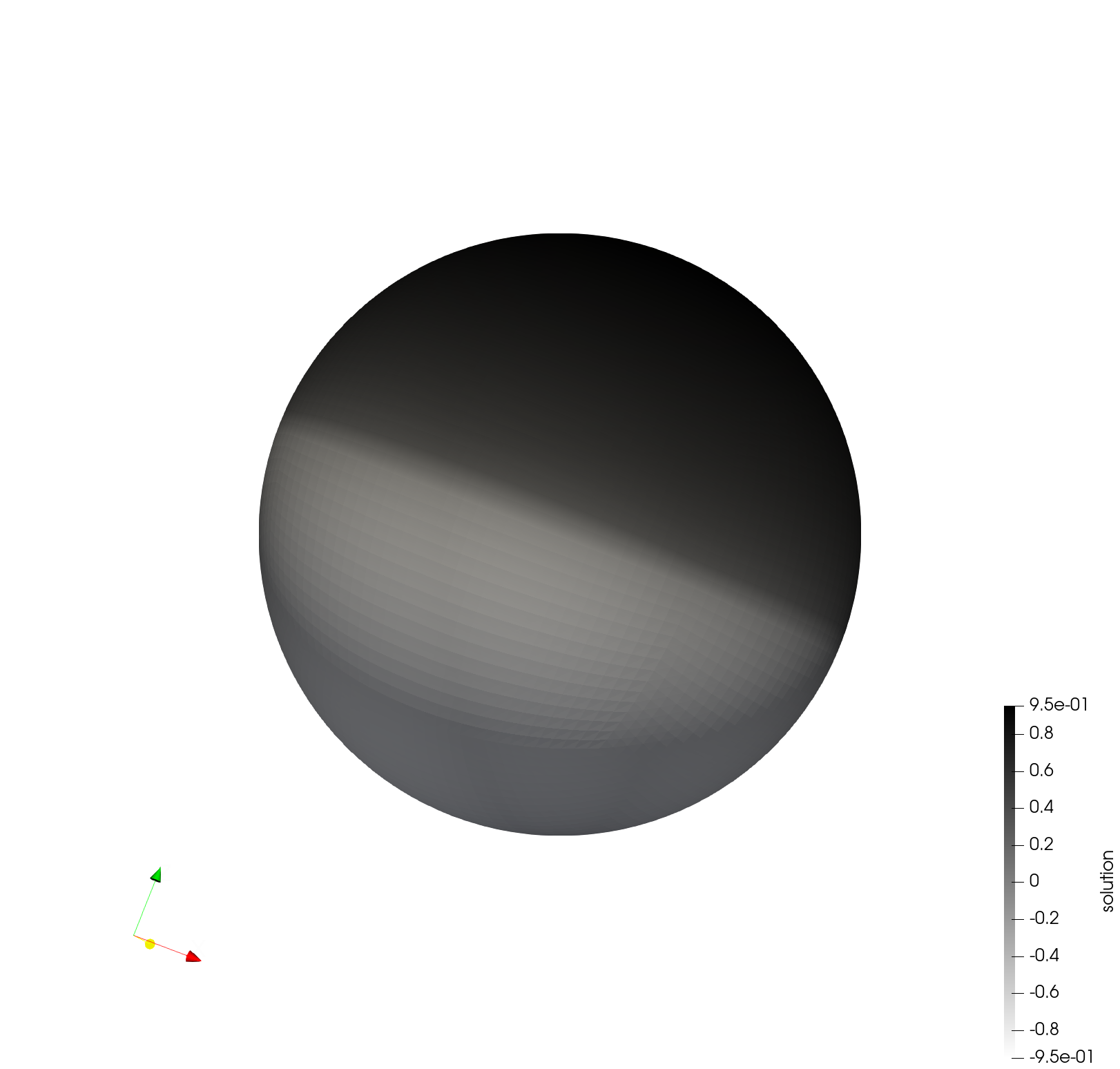} & \includegraphics[height=5cm]{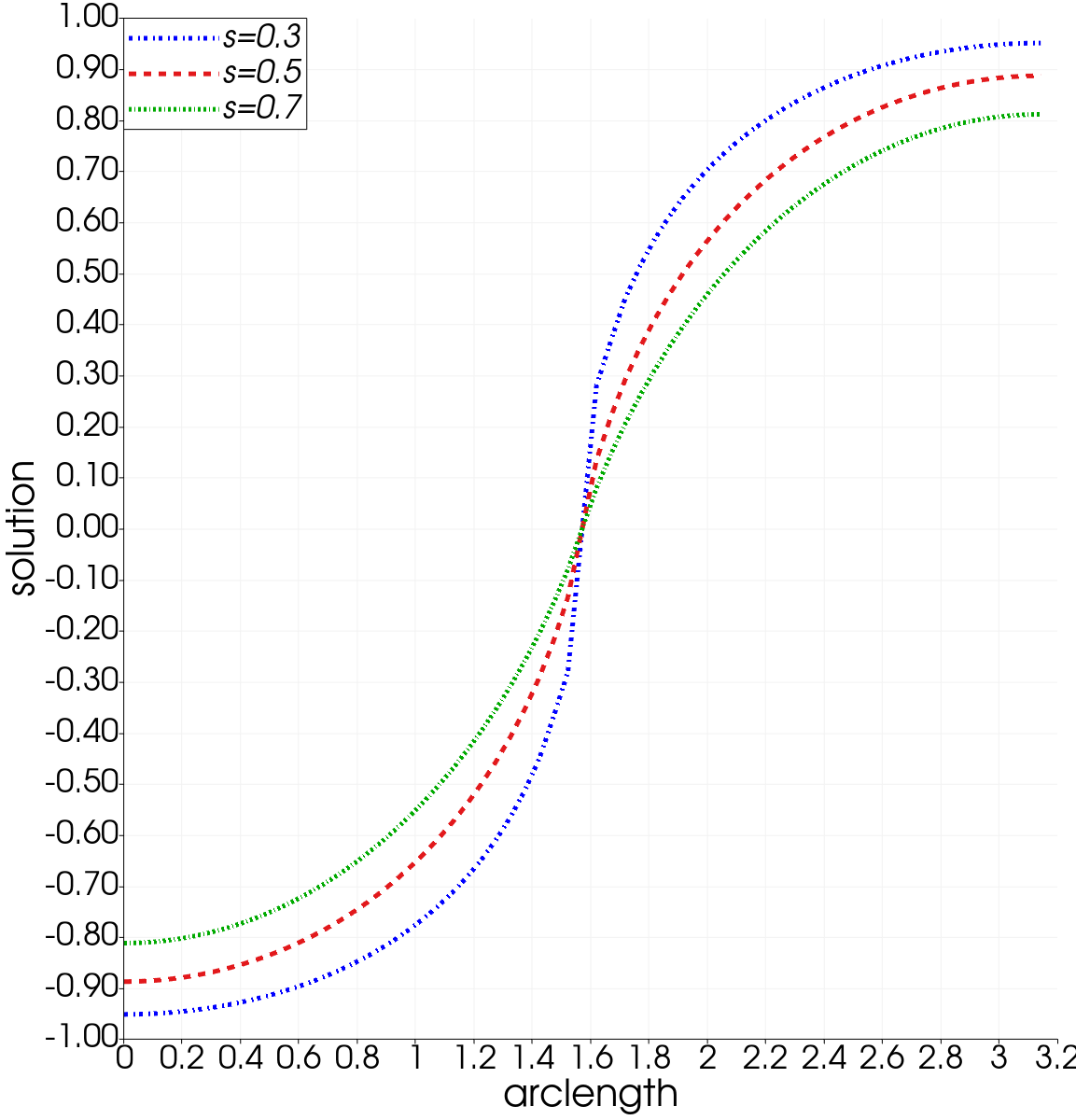} \\
		\end{tabular}
	\end{center}
	\caption{(Left) Approximated solution $U_k$ for $s=0.3$. The blue color corresponds to a value of -1 in $\widetilde f$ while the red color to 1. (Right) Values of the approximate solutions along the geodesic from the south pole to the north pole for $s=0.3$, $0.5$ and $0.7$. As $s$ increases the smoothing effect generated by the application of $L^{-s}$ increases as well.
		The visualizations are obtained using \texttt{Paraview} \cite{ayachit2015}.}
	\label{f:extract}
\end{figure}

We also consider the numerical scheme \eqref{e:sol_ukh} using a generic lifting operator $P_g$ instead of $P$ in the sub-problem \eqref{e:discrete-parametric}. To explicit $P_g$, we define for $i=1,2,3$
\[
	\mathcal D_i^+:=\{\bx \in \mathbb R^3 : x_i \ge |x_j|, j\neq i\}\text{ and }
	\mathcal D_i^-:=\{\bx \in \mathbb R^3 : x_i \le -|x_j|, j\neq i\}
\]
in such a way that $\mathcal D_i^\pm$ subdivide the unit sphere in six regions. Whence, we set $P_g\widetilde f = f\circ \mathbf P_g$, where the lift $\mathbf P_g$ is defined by the following piecewise lift onto $\gamma$: for $\bx \in \mathcal D_i^\pm$, set $\mathbf P_g^{i,\pm}(\bx) = \mathbf{z}$, where $z_j = x_j$ if $j\neq i$ and $z_i = \pm\sqrt{1-\sum_{k\neq i}x_k^2}$. Using this lifting operator to compute $U^\ell$ in \eqref{e:discrete-parametric} does not affect the convergence rate as reported in Figure~\eqref{f:convergence-generic-lift}.

\begin{figure}[hbt!]
	\begin{center}
		\begin{tabular}{cc}
			\includegraphics[scale=0.5]{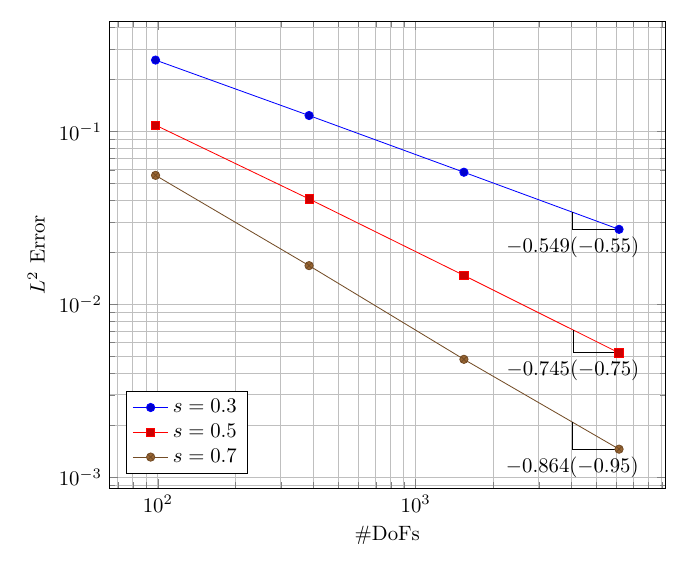} & \includegraphics[scale=0.5]{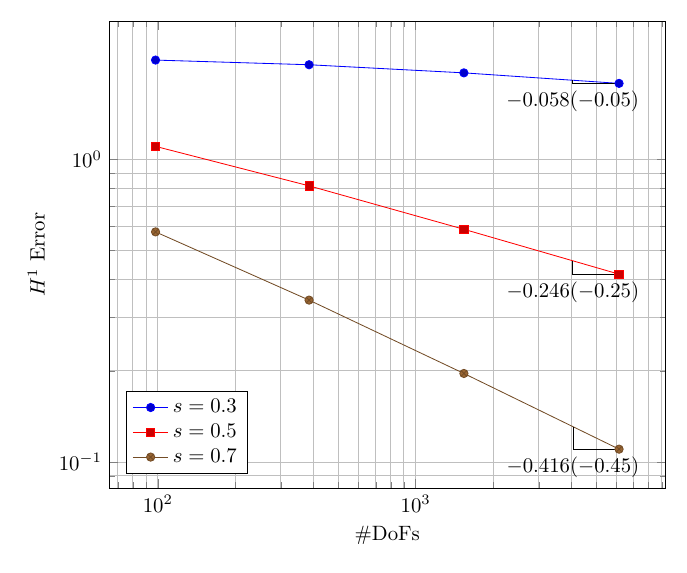}
		\end{tabular}
	\end{center}
	\caption{The errors $\| P_g\widetilde u-U_k\|_{L^2(\Gamma)}$ (left) and $\| P_g\widetilde u-U_k\|_{H^1(\Gamma)}$ (right)  versus the number of degree of freedoms for $s=0.3$, $0.5$ and $0.7$. For each error plot, the slope of the last segment together with the predicted convergence rate are also reported and matches the decay in Figure~\ref{f:convergence-parametric} (which uses the signed distance function).}
	\label{f:convergence-generic-lift}
\end{figure}

\subsubsection{Application to Gaussian Random Fields}
In this section, we discuss the numerical simulation of Gaussian random fields on a closed surface.
The latter is characterized as the solution $\widetilde u$ to the stochastic partial differential equation
\begin{equation}\label{e:spde}
	(\kappa^2 - \Delta_\gamma)^s \widetilde u = \mathcal W,\quad \text{on } \gamma,
\end{equation}
where $s>\frac4d,\kappa>0$ are regularity parameters and $\mathcal W$ denotes white noise. Note that $\mathcal W$ can be represented by the Karhumen-Lo\'eve expansion with respect to the orthonormal eigenbasis $\{\widetilde \psi_j\}_{j=1}^\infty$ in $L^2(\gamma)$, i.e.
$\mathcal W = \sum_{j=1}^\infty \xi_j\widetilde\psi_j$ with $\{\xi_j\}$ denoting a sequence of independent real-values standard normally distributed random variables. In \cite{jansson2021surface}, a surface finite element is proposed to approximate \eqref{e:spde} on the sphere with the truncated Karhumen-Lo\'eve expansion.

Here we follow the idea from \cite{bolin2020numerical} to approximate the white noise in the finite element space $\mathbb V(\mathcal T)$. Define $\mathcal W^{\Psi}:=\sum_{j=1}^N \xi_j \Psi_{j}$, where $N$ denotes the number of degrees of freedom for $\mathbb V(\mathcal T)$ and $\{\Psi_{j}\}_{j=1}^N$ is the set of eigenfunctions for the discrete Laplace-Beltrami operator. More precisely, $\Psi_{j}\in \mathbb V(\mathcal T)$ satisfies
\[
	\int_\Gamma \nabla_\Gamma \Psi_{j}\cdot\nabla_\Gamma V = \lambda_j \int_\Gamma \Psi_{j} V .
\]
To avoid the computation of discrete eigenfunctions, we can alternatively calculate $\mathcal W^{\Phi} = \sum_{j=1}\eta_j \Phi_j$. Here $\{\Phi_j\}$ are the shape functions on the vertices and $(\eta_1,\ldots,\eta_N)^T = R^{-1}(\xi_1,\ldots,\xi_N)^T$, where $R^\texttt{T} R = M$ with $M$ denoting the mass matrix with the entries $M_{ij} = \int_\Gamma \Phi_j\Phi_i$. Lemma~2.8 of \cite{bolin2020numerical} shows that $\mathbb E[\|\mathcal W^\Psi-\mathcal W^\Phi\|^2_{L^2(\Gamma)}] = 0$. In practice, the data vector for the discrete system should be $\mathbf G \mathbf z$ with $\mathbf z\sim \mathcal N(\mathbf 0,\mathbf I_{N\times N})$ with $\mathbf G$ denoting a Cholesky factor of the mass matrix.

Applying our numerical method in Section~\ref{ss:parametric}, we report a discrete solution to the problem \eqref{e:spde} in Figure~\ref{f:spde}. Here $s=0.75$, $\kappa=0.5$, and the surface $\gamma$ is a torus with the following parametrization: for $\theta,\phi\in [0,2\pi)$
\[
	\begin{aligned}
		x(\theta,\phi)  & = (2+0.5\cos \theta) \cos\phi , \\
		y(\theta, \phi) & = (2+0.5\cos\theta)\sin\phi,    \\
		z(\theta,\phi)  & = 0.5\sin\theta .
	\end{aligned}
\]
\begin{figure}[hbt!]
	\begin{center}
		\begin{tabular}{cc}
			\includegraphics[height=5.5cm]{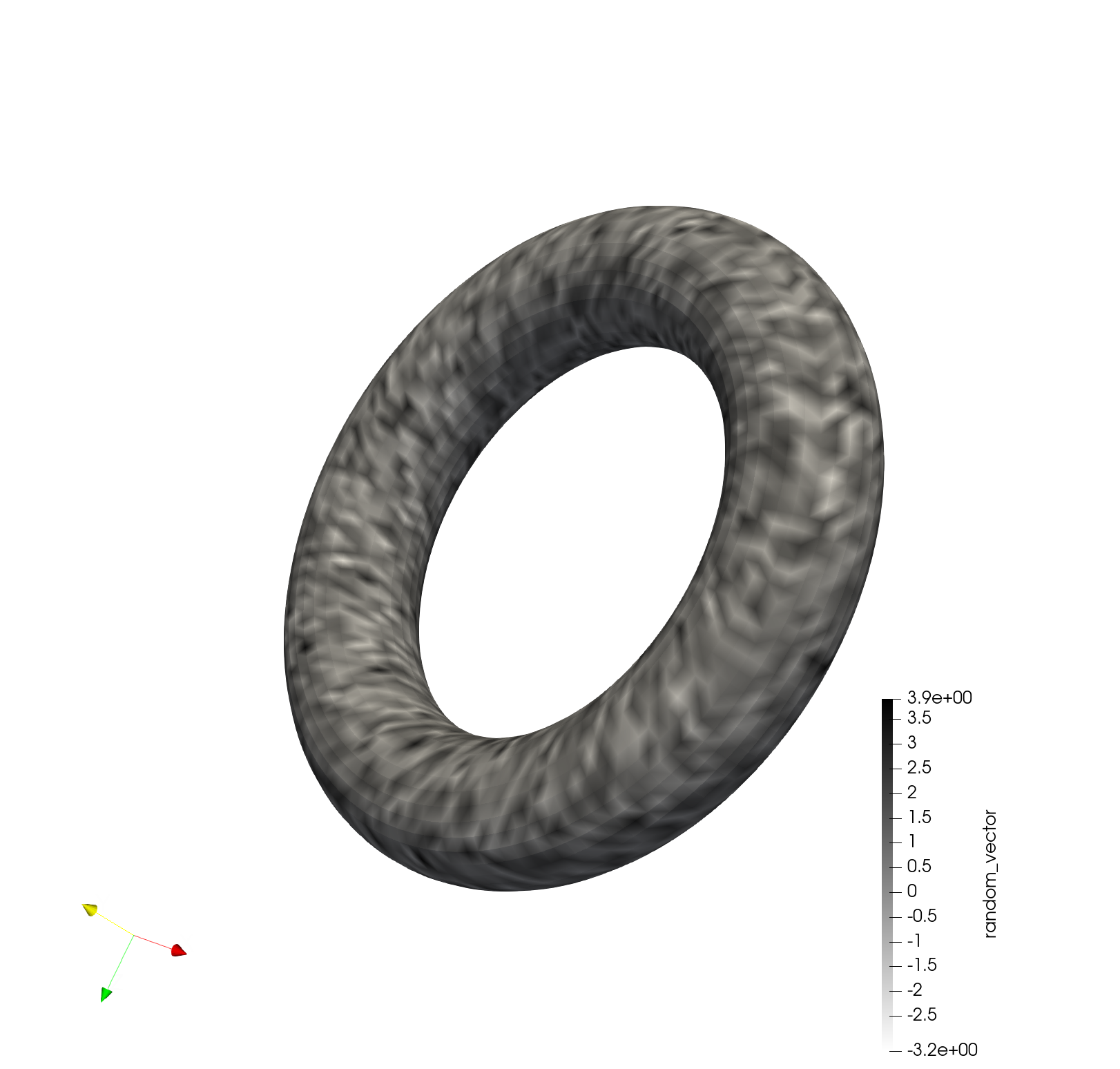} & \includegraphics[height=5.5cm]{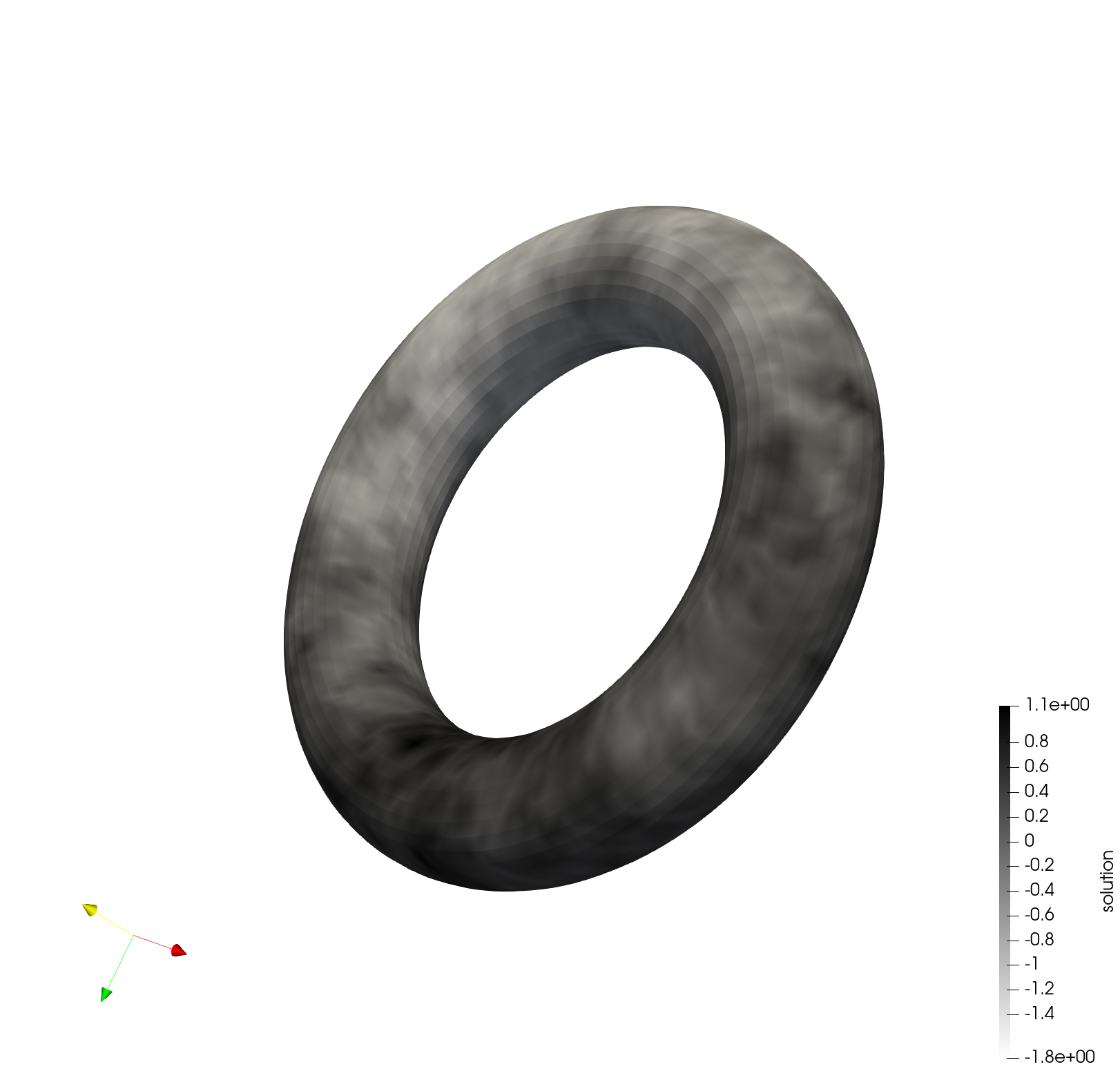} \\
		\end{tabular}
	\end{center}
	\caption{
		(Left) The random vector $\mathbf z$ and (right) the corresponding approximated random field sample $U_k$ for \eqref{e:spde} with $s=0.75$ and $\kappa=0.5$.
	}
	\label{f:spde}
\end{figure}
\section{Error Estimates}\label{s:analysis}
In this section, we provide $L^2$ and $H^1$ estimates for the discrepancy between the solution $P\widetilde u=P L^{-s}\widetilde f$ and its fully discrete approximation $U_{k}=\cQ_{k}^{-s}(L_{\mathcal T})\Ps\widetilde f$.
We discuss the error associated with the sinc quadrature first and move to the finite element error in a second step.

\subsection{Exponentially Convergent Sinc Quadrature}\label{ss:sinc}

We start by proving the stability of $\cQ^{-s}_k$.

\begin{lemma}[Stability]\label{l:sinc-stable}
	Let $2s\leq t \leq 2+2s$. Assume that $\widetilde f \in H^{t-2s}_\#(\gamma) \subset L^2_\#(\gamma)$.
	Then for any $\epsilon>0$, the sinc approximation $\widetilde u_k$ given by \eqref{e:quadrature} satisfies
	\[
		\|\widetilde u_k\|_{\dH^{t-\epsilon} (\gamma)}  \lesssim \max(1,\epsilon^{-1})\|\widetilde f\|_{H^{t-2s} (\gamma)} .
	\]
\end{lemma}
\begin{proof}
	We start with expression \eqref{e:quadrature} for $\widetilde u_k$ and estimate
	\begin{equation}\label{e:estim_ukk}
		\| \widetilde u_k\|_{\dH^{t-\epsilon} (\gamma)} \lesssim k \sum_{\ell = -M}^N e^{(1-s)y_\ell} \|\widetilde u^\ell\|_{\dH^{t-\epsilon}(\gamma)}.
	\end{equation}
	We now obtain bounds for the inner problems $\widetilde u^\ell := (\mu_\ell I + L)^{-1} \widetilde f$, where $\mu_\ell :=e^{y_\ell}$.
	Using the definition of the dotted norms \eqref{e:inter-def}, we deduce that
	\begin{equation}\label{e:uell_estim}
		\| \widetilde u^\ell  \|_{\dH^{t-\epsilon}(\gamma)}^2  =   \| (\mu_\ell I + L)^{-1} L^{(t-\epsilon)/2}\widetilde f  \|_{L^2(\gamma)}^2
		\lesssim \sum_{j=1}^\infty \widetilde \lambda_j^{t-2s}  \Big(\frac{\widetilde \lambda_j^{s-\epsilon/2}} {\widetilde \lambda_j +\mu_\ell} \Big)^{2}  |\widetilde f_j|^2,
	\end{equation}
	where  $\widetilde f_j :=  \langle \widetilde f,\widetilde \psi_j \rangle_{H^1_\#(\gamma)}$.
	Note that $ \frac{\widetilde \lambda_j^{s-\epsilon/2}} {\widetilde \lambda_j +\mu_\ell} \lesssim \min(1,\mu_\ell^{s-1-\epsilon/2})$ because
	$$
		\frac{\widetilde \lambda_j^{s-\epsilon/2}} {\widetilde \lambda_j +\mu_\ell} \leq \widetilde \lambda_j^{s-1-\epsilon/2} \lesssim 1
	$$
	and, using a Young inequality,
	$$
		\frac{\widetilde \lambda_j^{s-\epsilon/2}} {\widetilde \lambda_j +\mu_\ell} =
		\mu_\ell^{s-1-\epsilon/2}\frac{\mu_\ell^{1-s+\epsilon/2}\widetilde \lambda_j^{s-\epsilon/2}} {\widetilde \lambda_j +\mu_\ell}
		\leq \mu_\ell^{s-1-\epsilon/2}.
	$$
	As a consequence, \eqref{e:uell_estim} yields
	$$
		\| \widetilde u^\ell  \|_{\dot H^{t-\epsilon}(\gamma)}^2 \lesssim  \min(1,\mu_\ell^{s-1-\epsilon/2})  \| \widetilde f  \|_{\dot H^{t-2s}(\gamma)}^2 \lesssim \min(1,\mu_\ell^{s-1-\epsilon/2})  \| \widetilde f  \|_{H^{t-2s}(\gamma)}^2,
	$$
	where for the last estimate we used the equivalence of norms \eqref{e:norm_equi_dotted}.
	Returning to \eqref{e:estim_ukk}, we obtain the desired estimate
	\[
		\begin{split}
			\|\widetilde{u}_k\|_{\dH^{t-\epsilon}(\gamma)} &\lesssim \|\widetilde f\|_{H^{t-2s}(\gamma)} k\bigg(\sum_{\ell = -M}^1 e^{(1-s)y_\ell} + \sum_{\ell = 0}^N e^{-\epsilon y_\ell/2} \bigg)  \\
			&\lesssim \max(1,\epsilon^{-1}) \|\widetilde f\|_{H^{t-2s}(\gamma)}.
		\end{split}
	\]
\end{proof}
We remark that except for the factor $\epsilon$, the above stability is expected since $\widetilde u := (-\Delta_\gamma)^{-s} \widetilde f \in \dH^t(\gamma)$ for $\widetilde f \in H_\#^{t-2s}(\gamma)$ and thus it reads
$$
	\| \widetilde u_k \|_{\dH^{t-\epsilon}(\gamma)} \lesssim \max(1,\epsilon^{-1}) \| \widetilde u \|_{\dH^t(\gamma)}.
$$
In addition, when $t\leq 2$ the dotted spaces can be replaced by the regular interpolation spaces.
However, the  analysis below requires such estimate for $t > 2$.

Regarding the convergence of $\cQ^{-s}_k(L)$ towards $L^{-s}$, it follows from Theorem 3.2 and Remark 3.1 in \cite{BLP17}:
Given $\widetilde f\in \HHs^{2t}(\gamma)$ with $t\in [0,1]$ and $-t\le r<s$, there holds
\begin{equation}\label{e:sinc-error}
	\|(L^{-s}-\cQ^{-s}_k(L))\widetilde f\|_{\HH^{2(r+t)}(\gamma)}  \lesssim \rho(k,r,t)\|\widetilde f\|_{\HH^{2t}(\gamma)} ,
\end{equation}
where
\begin{equation}\label{e:rho}
	\rho(k,r,t) := \frac{e^{-\pi^2/(2k)}}{\sinh{(\pi^2/(2k))}} + e^{-(s-r^+)Nk} + e^{-(1-s)Mk}, \qquad r^+ := \max(0,r) .
\end{equation}
Balancing three terms in the definition of $\rho(k,r,t)$ yields
	\[
		\rho(k,r,t)  \sim C\bigg(\frac{1}{s-r^+}+\frac{1}{1-s}\bigg) e^{-\pi^2/(2k)} ,\quad\text{ as } k\to 0 ,
	\]
	where the constant $C$ is independent of $k$ and $s$; see Remark~3.3 in \cite{BLP17}.

\subsection{Space Discretization}\label{ss:fem}

The next theorem assesses the discrepancy between the sinc quadrature approximation (lifted to $\Gamma$) $P\widetilde u_k=P \cQ_k^{-s}(L)\widetilde f$  and the fully discrete approximation $U_{k} \in \vh$.
Its proof relies on several intermediate results and is therefore postponed to Section~\ref{ss:proof_main}.

\begin{theorem}[space discretization]\label{t:rate}
	Let $0\le\delta\le 1$. Assume that $\widetilde f\in \HHs^{2\delta}(\gamma)$.
	Let $\widetilde u_k = \cQ_k^{-s}(L)\widetilde f$ be the sinc quadrature approximation \eqref{e:quadrature} of $U_{k} \in \vh$ in turn given by \eqref{e:sol_ukh}.  Assume that $h$ is sufficiently small so that $\Gamma \subset \mathcal N$ and $h\leq e^{-1}$. Then there holds
	$$
		\|P \widetilde u_k - U_{k}\|_{L^2(\Gamma)} \lesssim \varepsilon(h)\|\widetilde f\|_{\HH^{2\delta}(\gamma)},
	$$
	where
	\begin{equation}\label{e:eps}
		\varepsilon(h) := \left\{
		\begin{aligned}
			 & h^2,                          & \quad \text{if } \delta+s > 1,  \\
			 & \ln(h^{-1})h^{2(\delta + s)}, & \quad\text{if } \delta+s \le 1. \\
		\end{aligned}
		\right .
	\end{equation}
	If in addition $\delta + s>\tfrac12$, then we have
	$$
		\|P \widetilde u_k - U_{k}\|_{H^1(\Gamma)}
		\lesssim \varepsilon(h)h^{-1}\|\widetilde f\|_{\HH^{2\delta}(\gamma)}.
	$$
\end{theorem}

When combing the error estimates from the sinc approximation and space discretization, we obtain the following corollary for the total error.
\begin{corollary}[total error]\label{c:rate}
	Let $0\le\delta\le 1$. Assume that $\widetilde f\in \HHs^{2\delta}(\gamma)$.
	Let $u = L^{-s}\widetilde f$ and $U_{k} \in \vh$ its approximation given by \eqref{e:sol_ukh}.
	Assume that $h$ is sufficiently small so that $\Gamma \subset \mathcal N$ and $h\leq e^{-1}$. Then there holds
	$$
		\|P \widetilde u - U_{k}\|_{L^2(\Gamma)} \lesssim (\varepsilon(h)+ \rho(k,1-\delta,\delta))\|\widetilde f\|_{\HH^{2\delta}(\gamma)},
	$$
	where $\varepsilon(h)$ is given by \eqref{e:eps} and $\rho(k,1-\delta,\delta)$ by \eqref{e:rho}.
	If in addition $\delta + s>\tfrac12$ then we have
	$$
		\|P \widetilde u - U_{k}\|_{H^1(\Gamma)}
		\lesssim (\varepsilon(h)h^{-1} + \rho(k,\tfrac12-\delta,\delta))\|\widetilde f\|_{\HH^{2\delta}(\gamma)}.
	$$
\end{corollary}
The proof of this corollary is also postponed to Section~\ref{ss:proof_main}.

\subsubsection{$L^2(\Gamma)$-orthogonal projection}\label{ss:l2}
The $L^2(\Gamma)$ orthogonal projection $\pi_{\mathcal T}:L^2(\Gamma)\rightarrow \vh$ onto $\vh$ will be instrumental in the analysis below.
For $v\in L^2(\Gamma)$, it is defined by the relations
\begin{equation}\label{e:L2proj}
	\int_{\Gamma} (\pi_{\mathcal T} v -v) W = 0, \qquad \forall W \in \vh
\end{equation}
and have similar properties as the standard $L^2$ projection on Euclidean domain.
We summarize the properties needed for the analysis in the following proposition.

%It is stable in both $L^2_\#(\Gamma)$ and $\cH^1(\Gamma)$ with stability constants depending on the quasi-uniformity and valence constants $c_q$ and $c_v$ defined in Section~\ref{ss:method}.

%following the idea from \cite{BX91}. Here we summarize the following approximation properties for $\pi_{\mathcal T}$ together with the estimates \eqref{i:norm-equivalency-ba}. The proof is postponed to Appendix.

\begin{proposition}\label{p:l2-proj-approx}
	The $L^2$ projection $\pi_{\mathcal T}$ is both $L^2$ and $H^1$ stable and in particular for $v\in H^r(\Gamma)$, $r\in[0,1]$, there holds
	\begin{equation}\label{e:stabl2}
		\|\pi_{\mathcal T} v\|_{H^r(\Gamma)}\lesssim \|v\|_{H^r(\Gamma)}.
	\end{equation}
	Consequently, for $0\leq r,t \leq 1$ and $\widetilde v\in\HH^{r+t}(\gamma)$ we have
	\begin{equation}\label{i:l2-proj-approx}
		\|(I-\pi_{\mathcal T})P\widetilde v\|_{H^t(\Gamma)}  \lesssim h^{r}\|\widetilde v\|_{\HH^{r+t}(\gamma)}.
	\end{equation}

	% 	Given $r\in[0,1]$ and $v\in H^{1+r}(\gamma)$,
	% 	\begin{equation}\label{i:l2-approx}
	% 		\|(I-\pi_{\mathcal T})P v\|_{L^2(\Gamma)}+h\|(I-\pi_{\mathcal T})P v\|_{H^1(\Gamma)}
	% 		\lesssim h^{1+r}\|v\|_{H^{1+r}(\gamma)} .
	% 	\end{equation}
	% 	Here we note that the hidden constant depends on both $c_v$ and $c_q$. If $v\in L^2(\gamma)$, we have
	% 	\begin{equation}\label{i:l2-stable}
	% 		\|\pi_{\mathcal T} P v\|_{L^2(\Gamma)}\lesssim \|v\|_{L^2(\gamma)} .
	% 	\end{equation}
\end{proposition}

The proof of the above results does not differ from the standard Euclidean case.
In fact, the $L^2$ stability and error estimates directly follows from the definition of the $L^2$ projection and the approximation properties of the Scott-Zhang interpolant \cite{scott1990finite}. The $H^1$ stability and error  estimates can be derived using the Bramble-Xu trick \cite{BX91}, which also require the approximation properties of the Scott-Zhang interpolant along with those of the cell-wise $L^2$ projection.
Interpolating the stability results between $H^1$ and $L^2$ yield \eqref{e:stabl2} while
interpolating the approximation results between $H^1$ and $L^2$ yield \eqref{i:l2-proj-approx}.
It is worth mentioning that the constant hidden in the $``\lesssim''$ signs depends on the quasi-uniformity constant $c_q$, the Jacobian constant $c_J$ and the maximum valence constant $c_v$ all defined in Section~\ref{ss:parametric}. The details are omitted for brevity.

\subsection{Operator Representation of the Discrete Approximations}

As in  Section~\ref{ss:LB}, we rewrite \eqref{e:sol_ukh} in operator form by defining the discrete counter-parts of $T$ and $L$.
We denote by $T_{\mathcal T} : \vh \cap L^2_\#(\Gamma) \to \vh \cap L^2_\#(\Gamma)$ the discrete solution operator defined by $T_{\mathcal T} G:= W$, where $W \in \vh \cap L^2_\#(\Gamma)$ is the unique solution (again guaranteed by Lax-Milgram) of
\begin{equation}\label{e:discrete-problem}
	a_{\Gamma}(W,V):=\int_{\Gamma} \nabla_{\Gamma} W \cdot \nabla_{\Gamma}V= \int_{\Gamma} G V ,\qquad \Forall V\in\vh.
\end{equation}
Its inverse on $\vh\cap L^2_\#(\Gamma)$ is denoted $L_{\mathcal T}$, i.e. $L_{\mathcal T}:=T_{\mathcal T}^{-1}:\vh\cap L^2_\#(\Gamma) \rightarrow \vh\cap L^2_\#(\Gamma)$. Note that $L_{\mathcal T}$ is invertible on $\vh \cap L^2_\#(\Gamma)$ but it will be convenient for the analysis below to extend its action on the whole $\vh$ ($L_{\mathcal T} 1 =0$).
Regardless, for $\widetilde f \in L^2_\#(\gamma)$, the solution $U_k \in \vh \cap L^2_\#(\Gamma)$ given by \eqref{e:sol_ukh} satisfies
\begin{equation}\label{e:dfrac}
	U_k = \cQ_k^{-s} (L_{\mathcal T}) \Ps \widetilde f := \frac{k\sin(\pi s)}{\pi}\sum_{\ell=-M}^N  e^{(1-s) y_\ell} (e^{y_\ell}I+  L_{\mathcal T})^{-1} \pi_{\mathcal T} \Ps \widetilde f,
\end{equation}
where $\pi_{\mathcal T}$ is the $L^2(\Gamma)$ orthogonal projection onto $\mathbb V(\mathcal T)$ (see Section~\ref{ss:l2}) and
$$
	\Ps:L^2_\#(\gamma) \rightarrow L^2_\#(\Gamma), \quad \textrm{is given by }\Ps \widetilde g = \sigma P \widetilde g.
$$
In particular, we observe that $\pi_{\mathcal T}\Ps\widetilde f \in \vh \cap L^2_\#(\Gamma)$ when $\widetilde f \in L^2_\#(\gamma)$.

Next we shall collect some instrumental estimates corresponding to discrete versions of \eqref{i:bound1} and \eqref{e:bound}.
We gather both results in the next lemma.
\begin{lemma}\label{l:bound-discrete}
	For $\mu \in (0,\infty)$, we have
	\begin{equation}\label{i:mu-bound0}
		\|(\mu I+L_{\mathcal T})^{-1} G\|_{L^2(\Gamma)} \lesssim \mu^{-1}\|G\|_{L^2(\Gamma)}, \qquad \forall G \in \vh.
	\end{equation}
	In addition, for $0\le r\le t\le 1$ and $\mu\in(0,\infty)$, there holds
	\begin{equation}\label{e:bound-discrete}
		\|L_{\mathcal T} (\mu I+L_{\mathcal T})^{-1} G\|_{H^{r}(\Gamma)}\lesssim
		\mu^{(r-t)/2} \|G\|_{H^{t}(\Gamma)}, \qquad \forall G \in \vh.
	\end{equation}
\end{lemma}
\begin{proof}
	For $G \in \vh$, we set $W:=(\mu I+L_{\mathcal T})^{-1} G \in \vh$ which satisfies
	$$
		\mu \int_\Gamma W V + a_\Gamma(W,V) = \int_{\Gamma} G V, \qquad \forall V \in \vh.
	$$
	With this,  \eqref{i:mu-bound0} readily follows by choosing $V=W$.

	To derive \eqref{e:bound-discrete}, we start with the cases $(r,t)=(0,0)$, $(1,1)$ and $(0,1)$.
	Thanks to the stability of the $L^2$-projection $\pi_{\mathcal T}$  \eqref{e:stabl2}, the discrete dotted norms $\| L_{\mathcal T}^{r/2} . \|_{L^2(\Gamma)}$ are equivalent to the classical $H^r(\Gamma)$ norms $\| . \|_{H^r(\Gamma)}$ on the discrete space $\vh$, i.e.
	$$
		\| L_{\mathcal T}^{r/2} V \|_{L^2(\Gamma)} \sim \| V \|_{H^r(\Gamma)}, \qquad \forall V \in \vh, \qquad r \in [0,1].
	$$
	We refer for e.g. to \cite{BP15} for a proof.
	Whence, the choice $V=L_{\mathcal T}W$ and the definition of $a_\Gamma$ in \eqref{e:discrete-problem} yield
	\begin{equation}\label{e:bounda}
		\mu \| L_{\mathcal T}^{1/2} W \|_{L^2(\Gamma)} + \| L_{\mathcal T} W \|_{L^2(\Gamma)} = a_\Gamma(G,W).
	\end{equation}
	This in  \eqref{e:bounda} implies that
	$$
		\mu \| W \|_{H^1(\Gamma)}^2 + \| L_{\mathcal T} W \|_{L^2(\Gamma)}^2 \lesssim a_\Gamma(G,W)
	$$
	so that two different Young inequalities lead to
	\begin{equation}\label{i:mu-bound1}
		\| L_{\mathcal T}(\mu I+L_{\mathcal T})^{-1} G \|_{L^2(\Gamma)} \lesssim \|L_{\mathcal T} W\|_{L^2(\Gamma)}
		\lesssim\|G\|_{L^2(\Gamma)}
	\end{equation}
	and
	\begin{equation}\label{e:boundb}
		\mu \|W\|_{H^1(\Gamma)}+\mu^{1/2}\|L_{\mathcal T} W\|_{L^2(\Gamma)} \lesssim \|G\|_{H^1(\Gamma)}.
	\end{equation}
	Relation \eqref{i:mu-bound1} and \eqref{e:boundb} are the desired estimates \eqref{e:bound-discrete} for $r=0$ and $t=0,1$, while \eqref{e:boundb} also implies that
	\begin{equation}\label{i:mu-bound3}
		\|L_{\mathcal T}  (\mu I+L_{\mathcal T})^{-1} G\|_{H^1(\Gamma)}
		\lesssim \|G\|_{H^1(\Gamma)} + \mu \|W\|_{H^1(\Gamma)}
		\lesssim\|G\|_{H^1(\Gamma)},
	\end{equation}
	which is \eqref{e:bound-discrete} for $r=1$ and $t=1$.

	The general results for $0\le r\le t\le 1$ follows by interpolation.
	The case $r=t$ case follows from the interpolation between the case $(r,t)=(0,0)$ and $(1,1)$, i.e.
	\eqref{e:boundb} and \eqref{i:mu-bound3} while the case $t=1$ ($r\in[0,1]$) is derived from the interpolation between \eqref{i:mu-bound1} and \eqref{e:boundb}. Interpolation between the case $r=t$ and the case $t=1$ gives \eqref{e:bound-discrete} in general.
\end{proof}

\subsection{Error Representation}\label{ss:error}

We now derive a representation of the difference $P\widetilde u_k - U_{k}$ instrumental for the analysis of the method.
First, we introduce $\pi_{\mathcal T} P\widetilde u_k$ and write
$$
	P\widetilde u_k - U_{k} = (P\widetilde u_k - \pi_{\mathcal T}P\widetilde u_k) + (\pi_{\mathcal T} P\widetilde u_k - U_k).
$$
The first difference $P\widetilde u_k - \pi_{\mathcal T} P \widetilde u_{k}$ concerns the approximation properties of the $L^2$ projection and can be estimated using Proposition~\ref{p:l2-proj-approx}.
The second difference is more problematic.
We start by deriving an error representation instrumental for the proposed analysis.

We expand the expression for  $\pi_{\mathcal T}P\widetilde u_k - U_{k}$ using \eqref{e:quadrature} and \eqref{e:dfrac} to arrive at
\begin{equation}\label{e:error_rep}
	\pi_{\mathcal T} P\widetilde u_k - U_{k} =  \frac{k\sin(\pi s)}{\pi}\sum_{\ell=-M}^N  \mu_\ell^{1-s}  \mathcal W_\ell \widetilde f,
\end{equation}
where $\mu_\ell:= e^{y_\ell}$ and
\begin{equation}\label{e:error_representation_A}
	\mathcal W_\ell \widetilde f:=\pi_{\mathcal T}P(\mu_{\ell}I+  L)^{-1}  \widetilde f -(\mu_{\ell}I+  L_{\mathcal T})^{-1} \pi_{\mathcal T}  \Ps \widetilde f.
\end{equation}
After several algebraic manipulations, we rewrite the above relation as
\begin{equation}\label{e:error_representation_B}
	\begin{split}
		\mathcal W_\ell \widetilde f& =(\mu_{\ell} I+L_{\mathcal T})^{-1} \left(  (\mu_{\ell} I+L_{\mathcal T})\pi_{\mathcal T} P  - \pi_{\mathcal T} \Ps (\mu_{\ell} I+L)\right)(\mu_{\ell} I+L)^{-1}\widetilde f\\
		& =L_{\mathcal T}(\mu_{\ell} I+L_{\mathcal T})^{-1} (\pi_{\mathcal T} P T-   T_{\mathcal T}\pi_{\mathcal T}\Ps)L(\mu_{\ell} I+L)^{-1} \widetilde f \\
		& \qquad + \mu_{\ell}(\mu_{\ell} I+L_{\mathcal T})^{-1}\pi_{\mathcal T}(P-\Ps)(\mu_{\ell} I+L)^{-1} \widetilde f=: \mathcal W^1_\ell  \widetilde f+ \mathcal W^2_\ell \widetilde f.
	\end{split}
\end{equation}
Note that the term $\pi_{\mathcal T} P T-   T_{\mathcal T}\pi_{\mathcal T} \Ps$ in $\mathcal W^1_\ell$ relates to the approximation of the solution to the  Laplace-Beltrami problem \eqref{e:dirichlet} by the finite element method \eqref{e:discrete-problem} while  $\pi_{\mathcal T} (P-\Ps)$ in $\mathcal W^2_\ell$ is of geometric nature and is dictated by the approximation of $\gamma$ by $\Gamma$.
We will estimate these two term separately.

\subsection{Geometric Approximation}
We recall that we assume that $\Gamma \subset \mathcal N$ and list several instrumental estimates regarding to approximation of $\gamma$ with $\Gamma$.
For their proofs, we refer for instance to \cite{bonito2020finite}.
We recall that thanks to the orthogonal property of $\mathbf P$, there exists a constant $c_\sigma$ only depending on $\gamma$ satisfying,
\begin{equation}\label{i:q-estimate}
	\|1-\sigma\|_{L^\infty(\Gamma)} \le c_\sigma h^2
\end{equation}
and $0<c\leq |\sigma(\bx)|\leq C$, $\bx \in \Gamma$, for some constants $c$ and $C$ only depending on $\gamma$.
Furthermore, the discrepancy between the bilinear forms $a_\gamma$ and $a_\Gamma$ satisfies
\begin{equation}\label{e:geom_bilinear}
	\int_{\Gamma} \nabla_{\Gamma} P\widetilde v \cdot \nabla_{\Gamma} P\widetilde w -
	\int_{\gamma} \nabla_{\gamma} \widetilde v \cdot \nabla_{\gamma} \widetilde w
	= \int_{\gamma} \nabla_{\gamma} \widetilde v \cdot {\mathbf E}\nabla_{\gamma} \widetilde w,
\end{equation}
where the error matrix ${\mathbf E}$ reflects the approximation of $\Gamma$ by $\gamma$ and satisfies
\begin{equation}\label{e:E}
	\|\mathbf E\|_{L^\infty(\gamma)}\lesssim h^2;
\end{equation}
see Lemma~15 and equation (56) in \cite{bonito2020finite}.

Thanks to the  $C^3$ regularity assumption on $\gamma$, the lift $\bP \in C^2(\mathcal N)^n$  and we have the equivalence of norms
$$
	\|P\widetilde v\|_{H^j(\Gamma)}\sim \|\widetilde v\|_{H^j(\gamma)}
$$
for all $\widetilde v \in H^j(\gamma)$, $j=0,1,2$.
By interpolation, we also deduce that for $r \in [0,2]$
\begin{equation}\label{i:norm-equivalency-ba}
	\|P\widetilde v\|_{H^r(\Gamma)}\sim \|\widetilde v\|_{H^r(\gamma)}, \qquad \forall  \widetilde v \in H^r(\gamma).
\end{equation}

%Here the constants hiding in the norm equivalency depend on the largest and smallest singular values of $\nabla d$. We can extend the above norm equivalency to $H^2$, where the corresponding constants depends on the largest and smallest singular values of the Hessian of $d$.

With these results, we are able to estimate $\mathcal W^2_\ell$ in the error representation \eqref{e:error_representation_B}.
\begin{lemma}\label{l:W2}
	For $\widetilde f \in L^2_\#(\gamma)$ we have
	$$
		\|\mathcal W^2_\ell \widetilde f \|_{H^j(\Gamma)} \lesssim\min(1,\mu_\ell^{-1}) h^{2-j} \| \widetilde f \|_{L^2(\gamma)}, \qquad j=0,1.
	$$
\end{lemma}
\begin{proof}
	We note that for $j=0,1$,
	\[
		\begin{aligned}
			\|\mathcal W^2_\ell \|_{L^2(\gamma)\to H^j(\Gamma)} & \lesssim \mu_\ell\|(\mu_\ell I + L_{\mathcal T})^{-1}\pi_{\mathcal T}\|_{L^2(\Gamma)\to H^j(\Gamma)}        \\
			                                                    & \qquad\|(1-\sigma) P\|_{L^2(\gamma)\to L^2(\Gamma)} \|(\mu_\ell I + L)^{-1}\|_{L^2(\gamma)\to L^2(\gamma)}  \\
			                                                    & \lesssim h^{-j} \mu_\ell\|(\mu_\ell I + L_{\mathcal T})^{-1}\pi_{\mathcal T}\|_{L^2(\Gamma)\to L^2(\Gamma)} \\
			                                                    & \qquad\|(1-\sigma) P\|_{L^2(\gamma)\to L^2(\Gamma)} \|(\mu_\ell I + L)^{-1}\|_{L^2(\gamma)\to L^2(\gamma)},
		\end{aligned}
	\]
	where we applied the inverse inequality to justify the second inequality. Using the geometric estimates \eqref{i:q-estimate} and \eqref{i:norm-equivalency-ba}, we find that
	\[
		\|(1-\sigma)P\|_{L^2(\gamma)\to L^2(\Gamma)}
		\lesssim \|1-\sigma\|_{L^\infty(\Gamma)} \|P\|_{L^2(\gamma)\to L^2(\Gamma)}
		\lesssim h^2
	\]
	and so
	$$
		\|\mathcal W^2_\ell \|_{L^2(\gamma)\to H^j(\Gamma)} \lesssim h^{2-j} \mu_\ell\|(\mu_\ell I + L_{\mathcal T})^{-1}\pi_{\mathcal T}\|_{L^2(\Gamma)\to L^2(\Gamma)} \|(\mu_\ell I + L)^{-1}\|_{L^2(\gamma)\to L^2(\gamma)}.
	$$
	This, together with   \eqref{i:mu-bound0}, \eqref{i:bound1} and the stability of the $L^2$ projection \eqref{e:stabl2} yields the desired estimate.
\end{proof}

\subsection{Error Estimate for the Laplace-Beltrami problem}
As already pointed out, the error representation \eqref{e:error_rep} requires estimates for the difference between $PT\widetilde f$,  the solution to Laplace-Beltrami problem \eqref{e:dirichlet} mapped to $\Gamma$ via $P$ and its finite element approximation $T_{\mathcal T} \pi_{\mathcal T} P\widetilde f$.
When the distortion is measured in $H^1$ or in $L^2$, Theorem~37 and~38 in \cite{bonito2020finite} guarantee that
\begin{equation}\label{i:error-lb}
	\|(P T-T_{\mathcal T} \pi_{\mathcal T} \Ps)\widetilde f\|_{L^2(\Gamma)}+h\|(P T-T_{\mathcal T} \pi_{\mathcal T} \Ps) \widetilde f\|_{H^1(\Gamma)}
	\lesssim h^2 \|\widetilde f\|_{L^2(\gamma)}.
\end{equation}
Using the approximation property of the $L^2$ projection \eqref{i:l2-proj-approx} and the regularity estimate \eqref{e:regularity-estim} for $s=1$ and $r=0$, we deduce a slight modification of the above error estimate
\begin{equation}\label{i:error-lb-2}
	\|(\pi_{\mathcal T} P T-T_{\mathcal T} \pi_{\mathcal T} \Ps)\widetilde f\|_{L^2(\Gamma)}+h\|(\pi_{\mathcal T}P T-T_{\mathcal T} \pi_{\mathcal T} \Ps) \widetilde f\|_{H^1(\Gamma)}
	\lesssim h^2 \|\widetilde f\|_{L^2(\gamma)}.
\end{equation}

The following proposition provides the error estimates in interpolation scales.

\begin{proposition}\label{p:approximation}
	Let $\widetilde f\in L^2_\#(\gamma)$ and $\alpha\in [0,1]$. Then there hold
	\begin{equation}\label{i:sol-l2proj-bound-h1}
		\|(\pi_{\mathcal T} P T-T_{\mathcal T} \pi_{\mathcal T}\Ps)\widetilde f\|_{H^1(\Gamma)} \lesssim h^{\alpha} \|\widetilde f\|_{\HH^{\alpha-1}(\gamma)}
	\end{equation}
	and
	\begin{equation}\label{i:sol-l2proj-bound-r}
		\|(\pi_{\mathcal T} P T-T_{\mathcal T}\pi_{\mathcal T}\Ps)\widetilde f\|_{H^{1-\alpha}(\mathcal T)}\lesssim h^{2\alpha}\|\widetilde f\|_{\HH^{\alpha-1}(\gamma)}.
	\end{equation}
\end{proposition}

We postpone the proof in Appendix~\ref{a:error} to focus on the estimation of $\mathcal W^1_\ell$ from the error representation \eqref{e:error_representation_B}.

\begin{lemma}\label{l:res}
	Let $\alpha\in [0,1]$ and $\delta \in [0,(1+\alpha)/2]$. For $\widetilde f\in \HHs^{2\delta}(\gamma)$ there holds
	$$
		\|\mathcal W^1_\ell \widetilde f\|_{L^2(\Gamma)}
		\lesssim \mu_\ell^{\alpha-\delta-1} h^{2\alpha}\|\widetilde f\|_{\HH^{2\delta}(\gamma)}.
	$$
	Furthermore, for $\alpha \in [\tfrac12,1]$ and $\delta \in [0, \alpha]$, we have
	$$
		\|\mathcal W^1_\ell \widetilde f\|_{H^1(\Gamma)}
		\lesssim \mu_\ell^{\alpha-\delta-1} h^{2\alpha-1}\|\widetilde f\|_{\HH^{2\delta}(\gamma)}.
	$$
\end{lemma}
\begin{proof}
	The $L^2$ estimate is derived upon composing the following three estimates
	\begin{align*}
		                                                                                               & \|L(\mu_\ell I+L)^{-1} \widetilde g \|_{\HH^{\alpha-1}(\gamma)}
		\leq  \mu_\ell^{-\delta+\frac{\alpha-1}{2}} \| \widetilde g \|_{\HH^{2\delta}(\gamma)}, \qquad &                                                                                                                                                                            & \forall \widetilde g \in \HH^{2\delta}_\#(\gamma)                                                     \\
		                                                                                               & \|(\pi_{\mathcal T} P T-   T_{\mathcal T}\pi_{\mathcal T}\Ps) \widetilde g \|_{H^{1-\alpha}(\Gamma)} \lesssim h^{2\alpha} \| \widetilde g\|_{H^{\alpha-1}(\gamma)}, \qquad &                                                   & \forall \widetilde g \in H^{\alpha-1}_\#(\gamma), \\
		                                                                                               & \|L_{\mathcal T}(\mu_\ell I+L_{\mathcal T})^{-1} G\|_{L^2(\Gamma)} \lesssim \mu_\ell^{\frac{\alpha-1}{2}}\| G \|_{H^{1-\alpha}(\Gamma)}, \qquad                            &                                                   & \forall G \in \vh.
	\end{align*}
	The first estimate follows from \eqref{e:bound} with $r=1-\alpha$ and $t=2\delta$, the second estimate is \eqref{i:sol-l2proj-bound-r}, and the third estimate follows from the discrete estimate \eqref{e:bound-discrete} with $r=0$ and $t=1-\alpha$.

	Similarly for the $H^1(\Gamma)$ estimate, we invoke the following three estimate
	\begin{align*}
		\|L(\mu_\ell I+L)^{-1} \widetilde g \|_{\HH^{2\alpha-2}(\gamma)}
		                                                                                              & \lesssim \mu_\ell^{-\delta+\alpha-1} \| \widetilde g \|_{\HH^{2\delta}(\gamma)}, \qquad &  & \forall \widetilde g \in \HH^{2\delta}_\#(\gamma), \\
		\|(\pi_{\mathcal T} P T-   T_{\mathcal T}\pi_{\mathcal T}\Ps) \widetilde g \|_{H^{1}(\Gamma)} & \lesssim h^{2\alpha-1} \| \widetilde g\|_{H^{2\alpha-2}(\gamma)}, \qquad                &  & \forall \widetilde g \in H^{2\alpha-2}_\#(\gamma), \\
		\|L_{\mathcal T}(\mu_\ell I+L_{\mathcal T})^{-1} G\|_{H^1(\Gamma)}                            & \lesssim \| G \|_{H^{1}(\Gamma)}, \qquad                                                &  & \forall G \in \vh,
	\end{align*}
	where this time the first estimate follows from \eqref{e:bound} with $r=2-2\alpha$ and $t=2\delta$, the second estimate is \eqref{i:sol-l2proj-bound-h1}, and the third estimate  follows from the discrete estimate \eqref{e:bound-discrete} with $r=t=1$.
\end{proof}

\subsection{Proof of Theorem~\ref{t:rate} and Corollary~\ref{c:rate}}\label{ss:proof_main}

We are now in a position to prove Theorem~\ref{t:rate} and Corollary~\ref{c:rate}.

\subsubsection{Proof of Theorem~\ref{t:rate}}
We only  prove the $L^2$ estimate with $\delta+s\le 1$ and comment on how to similarly derive all the other cases at the end of the proof.

As mentioned in the beginning of Section~\ref{ss:error}, we decompose the difference $P\widetilde u_k - U_k$ into two parts
$$
	P\widetilde u_k - U_k = (P\widetilde u_k - \pi_{\mathcal T}P\widetilde u_k) + (\pi_{\mathcal T} P\widetilde u_k - U_k)
$$
and estimate each term separately. We set $\as := s+\delta-\epsilon'$ for some $2\epsilon' < s+\delta$ to be chosen later.

\boxed{1} For the $L^2$ projection error term we have
$$
	\begin{aligned}
		\|(I  -\pi_{\mathcal T})  P\widetilde u_k\|_{L^2(\Gamma)}
		 & \lesssim h^{2\as}\|\widetilde u_k\|_{\HH^{2\delta+2s-2\epsilon}(\gamma)}                                                           \ \text{( \eqref{i:l2-proj-approx} with $t=0$, $r=2\delta+2s-2\epsilon'$)}   \\
		 & \lesssim h^{2\as}\|\widetilde u_k\|_{\dH^{2\delta+2s-2\epsilon}(\gamma)}                                         \                  (\textrm{norm equivalence } \eqref{e:norm_equi_dotted})                     \\
		 & \lesssim \frac{1}{\epsilon'}h^{2\as}\|\widetilde f\|_{H^{2\delta}(\gamma)}.                                         \     (\textrm{Lemma}~\ref{l:sinc-stable} \textrm{with }t=2\delta+2s, \epsilon=2\epsilon' )
	\end{aligned}
$$

\boxed{2} It remains to estimate $\| \pi_{\mathcal T} P\widetilde u_k-U_k\|_{L^2(\Gamma)}$. According to \eqref{e:error_rep}, we write
\begin{equation}\label{i:01}
	\|\pi_{\mathcal T} P \widetilde u_k-U_k\|_{L^2(\Gamma)}\lesssim k\sum_{\ell=-M}^N \mu_{\ell}^{1-s}
	\|\mathcal W_\ell \widetilde f\|_{L^2(\Gamma)}.
\end{equation}
Different arguments are needed depending on whether $\mu_{\ell}\in(0,1)$, $\mu_{\ell}\in [1,h^{-2\alpha^*/s})$ and $\mu_{\ell}\in [h^{-2\alpha^*/s},\infty)$.

\boxed{3} For $\mu_{\ell}\in [h^{-2\alpha^*/s},\infty)$, we use \eqref{e:error_representation_A} to write
$$
	\|\mathcal W_\ell \widetilde f\|_{L^2(\Gamma)}                                                          \\
	\lesssim \|\pi_{\mathcal T} P (\mu_{\ell} I+L)^{-1}\widetilde f\|_{L^2(\Gamma)}+\|(\mu_{\ell} I+L_{\mathcal T})^{-1}\pi_{\mathcal T} \Ps \widetilde f\|_{L^2(\Gamma)}.
$$
Estimates  \eqref{i:mu-bound0}, \eqref{i:bound1} together with the stability of the $L^2$-projection \eqref{e:stabl2} and the equivalence of norms \eqref{i:norm-equivalency-ba} guarantee that
$$
	\|\mathcal W_\ell \widetilde f\|_{L^2(\Gamma)}                                                          \\
	\lesssim \mu_{\ell}^{-1}\|\widetilde f\|_{L^2(\gamma)}.
$$
As a consequence, we arrive at the final estimate for this ranges of $\mu$
$$
	\begin{aligned}
		k\sum_{\mu_\ell\ge h^{-2\as/s}} & \mu_{\ell}^{1-s}\|\mathcal W\widetilde f\|_{L^2(\Gamma)}  \lesssim \|\widetilde f\|_{L^2(\gamma)} k\sum_{\mu_\ell \geq h^{-2\as/s}} \mu_{\ell}^{-s}                                                  \\
		                                & \le \|\widetilde f\|_{L^2(\gamma)} \frac{k \mu_{\ell_0}^{-s}}{1-e^{-ks}} \le  h^{2\alpha^*}\|\widetilde f\|_{L^2(\gamma)} \frac{k}{1-e^{-ks}} \lesssim h^{2\alpha^*} \|\widetilde f\|_{L^2(\gamma)},
	\end{aligned}
$$
where $\ell_0$ is the smallest integer satisfying $\mu_{\ell_0}\ge h^{-2\as/s}$. We also note that to derive the last inequality, we used the fact $\mu_\ell = e^{y_\ell}=e^{k\ell}$ and that for all $a>0$, $k/(1-e^{-ka}) \lesssim \tfrac1a$.

\boxed{4} For $\mu_{\ell} \in [1,h^{-2\alpha^*/s})$, we take advantage of the error representation \eqref{e:error_representation_B}, the estimates provided in Lemma~\ref{l:res} with $\alpha=\as\in[0,1]$ and Lemma~\ref{l:W2}  to write
$$
	\begin{aligned}
		k & \sum_{1 \leq \mu_\ell<h^{-2\as/s}}  \mu_{\ell}^{1-s}\|\mathcal W \widetilde f\|_{L^2(\Gamma)}                                  \\
		  & \lesssim h^{2\alpha^*}\|\widetilde f\|_{\HH^{2\delta}(\gamma)}
		k\sum_{1\leq \mu_\ell<h^{-2\as/s}}\mu_{\ell}^{-s+\as-\delta}
		+  h^{2}\|\widetilde f\|_{L^2(\gamma)}
		k\sum_{1 \leq \mu_\ell<h^{-2\as/s}}\mu_{\ell}^{-s}                                                                                 \\
		  & \lesssim  h^{2\as}\|\widetilde f\|_{\HH^{2\delta}(\gamma)}
		\frac{k }{1-e^{-k\epsilon'}} + h^{2}\frac{k }{1-e^{ks}} \|\widetilde f\|_{L^2(\gamma)}                                             \\
		  & \lesssim \frac 1 {\epsilon'} h^{2\as}\|\widetilde f\|_{\HH^{2\delta}(\gamma)} + \frac 1 s h^{2}\|\widetilde f\|_{L^2(\gamma)}.
	\end{aligned}
$$

\boxed{5} For $\mu_{\ell}\in (0,1)$, we proceed similarly as in the previous step invoking Lemma~\ref{l:res} with $\alpha=1$, $\delta=0$ and Lemma~\ref{l:W2} to get
$$
	k\sum_{0<\mu_\ell<1}\mu_{\ell}^{1-s}\|\mathcal W_\ell \widetilde f\|_{L^2(\Gamma)}
	\lesssim h^{2}\|\widetilde f\|_{L^2(\gamma)}k\sum_{0<\mu_\ell<1} \mu_{\ell}^{1-s}\,
	\lesssim h^{2\alpha^*}\|\widetilde f\|_{L^2(\gamma)} .
$$

\boxed{6} Gathering the estimates obtained at each step, we arrive at
$$
	\| P\widetilde u_k - U_k\|_{L^2(\gamma)} \lesssim \frac{h^{-2\epsilon'}}{\epsilon'} h^{2(\delta+s)} \|\widetilde f\|_{H^{2\delta}(\gamma)},
$$
which, upon setting $\epsilon' = (\delta+s)/(2\ln(h^{-1}))$ satisfying $0<2\epsilon' < \delta+s$ thanks to the assumption $h\leq 1/e$,  is the desired $L^2$ estimate for $\delta+s \le 1$.

\boxed{7} The $H^1$ error estimate is obtained identically except that an inverse estimate is used at the beginning of step 3 to get
$$
	\|\mathcal W_\ell \widetilde f\|_{H^1(\Gamma)}                                                          \\
	\lesssim h^{-1}\mu_{\ell}^{-1}\|\widetilde f\|_{L^2(\gamma)}.
$$
The case $\delta+s>1$ also follows using similar arguments, the noticeable differences being that $\as=1$ and thus the estimate is free from $\epsilon'$ terms.
The details are omitted.

\subsubsection{Proof of Corollary~\ref{c:rate}}
We start with the $L^2$-estimate.
The total error is decomposed into the sinc quadrature error and the finite element error
\[
	\|P \widetilde u - U_{k}\|_{L^2(\Gamma)} \le \|P (\widetilde u - u_k)\|_{L^2(\Gamma)}
	+ \|P \widetilde u_k - U_{k}\|_{L^2(\Gamma)}.
\]
For the sinc quadrature error, we invoke the norm equivalence  $\|P.\|_{H^j(\Gamma)}\sim \|.\|_{H^j(\gamma)}$ valid for $j=0,1$ (see \eqref{i:norm-equivalency-ba}) and the sinc quadrature error \eqref{e:sinc-error} with $t=\delta$ and $r=-\delta$.
Theorem~\ref{t:rate} provides the desired results for the space discretization error.
The $H^1$ error estimate follows similarly but upon invoking \eqref{e:sinc-error} with $t=\delta$ and $r=\tfrac12-\delta$.

\section{Conclusions}
We proposed a numerical scheme to approximate the negative powers of the Laplace-Beltrami operator by discretizing the outer integral of the Balakrishnan formula with a sinc quadrature and using a parametric finite element method to approximate the evaluation of the integrant at each quadrature point. We show that if the surface is of class $C^3$ and when the signed distance function is used as lift to define the parametric finite element method, the numerical scheme delivers the same rates of convergence as on Euclidean domains.

\begin{comment}
Two remarks on the numerical schemes are listed below.

\emph{Approximation diagram:} Combing the proof strategies in \cite{BLP17,BP15} as well as the proofs in Section~\ref{s:analysis}, we have completed the following approximation diagram:
\begin{center}
	\begin{tikzpicture}
		\matrix (m) [matrix of math nodes,row sep=3em,column sep=4em,minimum width=2em]
		{
			U_k & P\widetilde u_k                       \\
			U   & P\widetilde u\\};
		\path[-stealth]
		(m-1-1) edge node [left] {$k\to 0$} (m-2-1)
		edge node [above] {$h\to 0$} (m-1-2)
		(m-2-1.east|-m-2-2) edge node [below] {$h\to 0$}
		node [above] {} (m-2-2)
		(m-1-2) edge node [right] {$k\to 0$} (m-2-2);
	\end{tikzpicture}
\end{center}
where $U = L_{\mathcal T}^{-s}P_{\#}\widetilde f$. We have to point out that in order to show the convergence in $H^1(\Gamma)$-norm along the left-bottom routine, we require the stability of the $L^2$-projection $\pi_{\mathcal T}$ in $H^{r}(\Gamma)$ for $r\in(1,2]$. See Lemma~4.1 and Theorem~4.3 in \cite{BLP17} for more details. Here along the top-right routine, we avoid this property and utilize the stability of the sinc quadrature (i.e. Lemma~\ref{l:sinc-stable}) instead.
\end{comment}

Whether a generic lift can be used within the proposed algorithm remains an open question.
The arguments used in the analysis of the algorithm relies on the second order approximation property of the geometric quantities $\sigma$ and $\mathbf{E}$; see for instance Lemma~\ref{l:W2} and \eqref{e:alternate3} in Appendix.
For generic lifts, such as $\mathbf{P}_g$ defined in Section~\ref{s:simulation}, the geometric quantities are only approximated to first order.
To illustrate this fact, we report in Figure~\ref{f:sigma} the behavior of $\| \sigma - 1 \|_{L^\infty(\Gamma)}$ versus the number of degree of freedoms (which is equivalent to $h^{-1/n}$) for the lift $\mathbf{P}_g$. It is conceivable that the arguments provided in \cite{bonito2020finite} and in particular the technical proposition (Proposition~34 in \cite{bonito2020finite}) can be extended in this context to explain the convergence results obtained in Section~\ref{ss:numerical} when using $\mathbf{P}_g$.

\begin{figure}[hbt!]
	\begin{center}
		\includegraphics[scale=0.5]{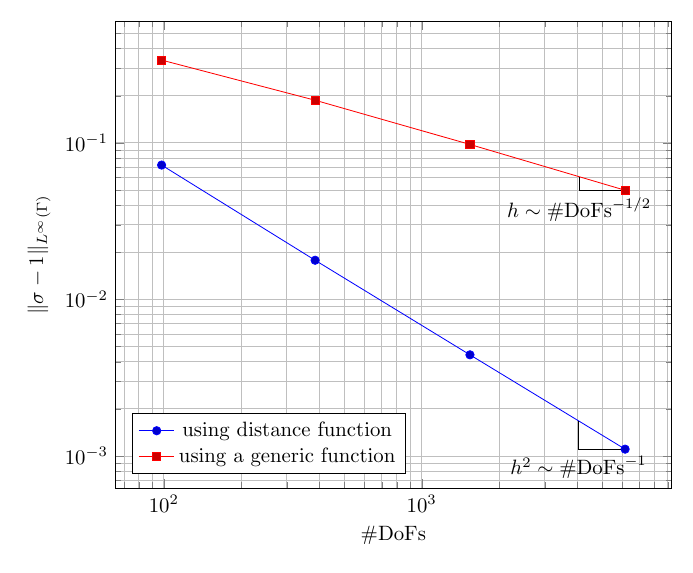}
	\end{center}
	\caption{Behavior $\| \sigma - 1\|_{L^\infty(\Gamma)}$ versus the number of degree of freedoms using $\mathbf{P}$ and $\mathbf{P}_g$. The space discretization setting is the same as we introduced in Section~\ref{s:simulation}. Here we approximate the norm by computing the maximum of $|\sigma - 1|$ at all 6th-order Gaussian quadrature points in each cell of $\mathcal T$.}
	\label{f:sigma}
\end{figure}

\appendix
\section{Proof of Proposition~\ref{p:approximation}}\label{a:error}
We start by noting that
\[
	\| (P T - T_{\mathcal T} \pi_{\mathcal T}\Ps) \widetilde f\|_{H^1(\Gamma)}\le \| P T \widetilde f\|_{H^1(\Gamma)} + \|T_{\mathcal T}  \pi_{\mathcal T} \Ps \widetilde f\|_{H^1(\Gamma)}
\]
so that the stability of the $L^2$ projection \eqref{e:stabl2}, the norm equivalence property \eqref{i:norm-equivalency-ba} and standard energy estimates yield
\begin{equation}\label{e:first_inter}
	\|( P T - T_{\mathcal T}\pi_{\mathcal T} \Ps)\widetilde f\|_{H^1(\Gamma)}\lesssim \| \widetilde f \|_{H^{-1}(\gamma)}.
\end{equation}
Whence, the interpolation between the above estimate and the standard $H^1$ error estimate \eqref{i:error-lb} for the Laplace-Beltrami problem imply that
$$
	\|( P T - T_{\mathcal T}\pi_{\mathcal T} \Ps) \widetilde f\|_{H^1(\Gamma)}\lesssim h^\alpha \| \widetilde f \|_{H^{\alpha-1}(\gamma)}.
$$
To derive  \eqref{i:sol-l2proj-bound-h1} it suffices to note that the approximation property of the $L^2$ projection \eqref{i:l2-proj-approx} for $t=1$ and $r=\alpha$, the equivalence of norms \eqref{i:norm-equivalency-ba}, and the regularity estimate \eqref{e:regularity-estim} guarantee that
$$
	\|( P T  - \pi_{\mathcal T} PT) \widetilde f \|_{H^1(\Gamma)} \lesssim h^\alpha \| PT \widetilde f\|_{H^{1+\alpha}(\Gamma)}
	\lesssim h^\alpha \|  \widetilde f\|_{H^{\alpha -1}(\gamma)}.
$$

To prove  \eqref{i:sol-l2proj-bound-r}, we argue again that in view of the approximation properties of the $L^2$ projection \eqref{i:l2-proj-approx}, it suffices to show that for $\alpha \in [0,1]$.
\begin{equation}\label{e:alternate2}
	\| \widetilde e \|_{H^{1-\alpha}(\gamma)}\lesssim h^{2\alpha} \|\widetilde f\|_{\HH^{\alpha-1}(\gamma)},
\end{equation}
where $\tilde e := (T - P^{-1}T_{\mathcal T} \pi_{\mathcal T}\Ps) \widetilde f$. Note that $\tilde e$ does not have vanishing mean value on $\gamma$.
However, we have
$$
	\bar{e} :=\frac1{|\gamma|}\int_\gamma \widetilde e
	= -\frac1{|\gamma|}\int_{\Gamma} T_{\mathcal T} \pi_{\mathcal T}\Ps \widetilde f \sigma
	= \frac1{|\gamma|}\int_{\Gamma} T_{\mathcal T} \pi_{\mathcal T}\Ps \widetilde f (1- \sigma)
$$
because $\int_\gamma T \widetilde f = \int_{\Gamma} T_{\mathcal T} \pi_{\mathcal T}\Ps \widetilde f = 0$.
As a consequence, in view of estimate \eqref{i:q-estimate} for $\| 1-\sigma\|_{L^\infty(\Gamma)}$, the equivalence of norms \eqref{i:norm-equivalency-ba} and  the estimate
\begin{equation}\label{e:inter2}
	\| T_{\mathcal T}  G \|_{H^1(\Gamma)} \lesssim \| G \|_{H^{-1}(\Gamma)}  \qquad \forall G \in \vh,
\end{equation}
we deduce that
$$
	\| \bar e \|_{H^{1-\alpha}(\gamma)} \lesssim | \bar e | \lesssim h^2 \|T_{\mathcal T} \pi_{\mathcal T}\Ps \widetilde f\|_{L^2(\Gamma)} \lesssim h^2 \| \widetilde f \|_{L^2(\gamma)}
$$
and so
\begin{equation}\label{e:alternate3}
	\| \widetilde  e \|_{H^{1-\alpha}(\gamma)} \lesssim \| \widetilde  e - \bar e  \|_{H^{1-\alpha}(\gamma)} + h^2 \| \widetilde f \|_{L^2(\gamma)}.
\end{equation}

To estimate $\| \widetilde  e - \bar e  \|_{H^{1-\alpha}(\gamma)}$, we borrow a duality argument proposed in \cite[Theorem~38]{bonito2020finite} by proceeding as follows. We compute
\[
	\begin{split}
		\|\widetilde e - \bar e\|_{\HH^{1-\alpha}(\gamma)} &= \sup_{\widetilde w\in\HH^{\alpha-1}_\#(\gamma)}
		\frac{\langle\widetilde e - \bar e, \widetilde w\rangle_{\HH^{1-\alpha}_\#(\gamma)}}{\|\widetilde w\|_{\HH^{\alpha-1}(\gamma)}} \\
		& = \sup_{\widetilde w\in \HHs^{\alpha-1}(\gamma)}
		\frac{a_{\gamma}(\widetilde e-\bar e, T\widetilde w)}{\|\widetilde w\|_{\HH^{\alpha-1}(\gamma)}} \\
		& = \sup_{\widetilde w\in \HHs^{\alpha-1}(\gamma)}
		\frac{a_{\gamma}(\widetilde e, T\widetilde w - P^{-1} T_{\mathcal T} \pi_{\mathcal T} P_\# T \widetilde w) + a_{\gamma}(\widetilde e, P^{-1} T_{\mathcal T}  \pi_{\mathcal T}P_\# T \widetilde w)}{\|\widetilde w\|_{\HH^{\alpha-1}(\gamma)}}.
	\end{split}
\]
For the first term, we invoke the norm equivalence property \eqref{i:norm-equivalency-ba} to write
$$
	a_{\gamma}(\widetilde e, T\widetilde w - P^{-1} T_{\mathcal T} \pi_{\mathcal T} P_\# T \widetilde w)
	\lesssim \|P\widetilde e\|_{H^1(\Gamma)}\|PT\widetilde w - T_{\mathcal T}  \pi_{\mathcal T}P_\# T \widetilde w\|_{H^1(\Gamma)}.
$$
This,  coupled with the estimate \eqref{e:first_inter} and the regularity estimate \eqref{e:regularity-estim} for $s=1$ and $r =\alpha-1$,   yield
\begin{equation*}
	\begin{split}
		a_{\gamma}(\widetilde e, T\widetilde w - P^{-1} T_{\mathcal T} \pi_{\mathcal T} P_\# T \widetilde w) &\lesssim h^{2r}\|\widetilde f\|_{\HH^{\alpha-1}(\gamma)}\|T\widetilde w\|_{\HH^{\alpha+1}(\gamma)} \\
		&\lesssim h^{2r}\|\widetilde f\|_{\HH^{\alpha-1}(\gamma)}\|\widetilde w\|_{\HH^{\alpha-1}(\gamma)}.
	\end{split}
\end{equation*}

For the second term, we use the definition of $\widetilde e$ and take advantage of the geometric estimate \eqref{e:geom_bilinear} relating the bilinear forms $a_\Gamma$ and $a_\gamma$ to show that
\[
	a_{\gamma}(\widetilde e,P^{-1} T_{\mathcal T} \pi_{\mathcal T}P_\# T \widetilde w) =  \int_{\gamma} \nabla_\gamma (P^{-1}T_{\mathcal T}\pi_{\mathcal T}\Ps \widetilde f)\cdot \mathbf E\nabla_\gamma (P^{-1} T_{\mathcal T} \pi_{\mathcal T}P_\# T \widetilde w),
\]
where $\| \mathbf E \|_{L^\infty(\gamma)} \lesssim h^2$ according to \eqref{e:E}.
Whence, the equivalence of norms \eqref{i:norm-equivalency-ba} and estimate \eqref{e:inter2} imply that
$$
	\begin{aligned}
		a_{\gamma}(\widetilde e,P^{-1} T_{\mathcal T} & \pi_{\mathcal T}P_\# T \widetilde w)                                                                                                                          \\
		                                              & \lesssim h^2 \|T_{\mathcal T}\pi_{\mathcal T}\Ps \widetilde f\|_{H^1(\Gamma)} \|  T_{\mathcal T} \pi_{\mathcal T}P_\# T \widetilde w\|_{H^1(\Gamma)} \lesssim
		h^2 \|\widetilde f\|_{L^2(\gamma)} \|  \widetilde w\|_{H^{\alpha-1}(\gamma)}.
	\end{aligned}
$$

Returning to $\| \widetilde e - \bar e\|_{H^{1-\alpha}(\gamma)}$ we obtain
$$
	\| \widetilde e - \bar e\|_{H^{1-\alpha}(\gamma)} \lesssim  h^{2r}\|\widetilde f\|_{\HH^{\alpha-1}(\gamma)}
$$
and thus \eqref{e:alternate2} follows from \eqref{e:alternate3}.
This ends the proof.

\bibliographystyle{plain}
\bibliography{FractionalLaplaceBeltrami}

\begin{thebibliography}{10}

\bibitem{arndt2020deal}
Daniel Arndt, Wolfgang Bangerth, Bruno Blais, Thomas~C Clevenger, Marc Fehling,
  Alexander~V Grayver, Timo Heister, Luca Heltai, Martin Kronbichler, Matthias
  Maier, et~al.
\newblock The deal. ii library, version 9.2.
\newblock {\em Journal of Numerical Mathematics}, 1(ahead-of-print), 2020.

\bibitem{ayachit2015}
U.~Ayachit.
\newblock {\em The ParaView Guide: A Parallel Visualization Application}.
\newblock Kitware, Inc., USA, 2015.

\bibitem{Balakrishnan60}
A.~V. Balakrishnan.
\newblock Fractional powers of closed operators and the semigroups generated by
  them.
\newblock {\em Pacific J. Math.}, 10:419--437, 1960.

\bibitem{banjai2020exponential}
Lehel Banjai, Jens~M Melenk, and Christoph Schwab.
\newblock Exponential convergence of $ hp $ fem for spectral fractional
  diffusion in polygons.
\newblock {\em arXiv preprint arXiv:2011.05701}, 2020.

\bibitem{bolin2020numerical}
David Bolin, Kristin Kirchner, and Mih{\'a}ly Kov{\'a}cs.
\newblock Numerical solution of fractional elliptic stochastic pdes with
  spatial white noise.
\newblock {\em IMA Journal of Numerical Analysis}, 40(2):1051--1073, 2020.

\bibitem{BP13}
A.~Bonito and J.E. Pasciak.
\newblock Numerical approximation of fractional powers of elliptic operators.
\newblock {\em Math. Comp.}, 84(295):2083--2110, 2015.

\bibitem{bonito2018numerical}
Andrea Bonito, Juan~Pablo Borthagaray, Ricardo~H Nochetto, Enrique Ot{\'a}rola,
  and Abner~J Salgado.
\newblock Numerical methods for fractional diffusion.
\newblock {\em Computing and Visualization in Science}, 19(5-6):19--46, 2018.

\bibitem{bonito2016high}
Andrea Bonito, J~Manuel Casc{\'o}n, Khamron Mekchay, Pedro Morin, and Ricardo~H
  Nochetto.
\newblock High-order afem for the laplace--beltrami operator: Convergence
  rates.
\newblock {\em Foundations of Computational Mathematics}, 16(6):1473--1539,
  2016.

\bibitem{bonito2013afem}
Andrea Bonito, J~Manuel Casc{\'o}n, Pedro Morin, and Ricardo~H Nochetto.
\newblock Afem for geometric pde: the laplace-beltrami operator.
\newblock In {\em Analysis and numerics of partial differential equations},
  pages 257--306. Springer, 2013.

\bibitem{bonito2019posteriori}
Andrea Bonito and Alan Demlow.
\newblock A posteriori error estimates for the laplace--beltrami operator on
  parametric {$C^2$} surfaces.
\newblock {\em SIAM Journal on Numerical Analysis}, 57(3):973--996, 2019.

\bibitem{bonito2020finite}
Andrea Bonito, Alan Demlow, and Ricardo~H Nochetto.
\newblock Finite element methods for the laplace--beltrami operator.
\newblock In {\em Handbook of Numerical Analysis}, volume~21, pages 1--103.
  Elsevier, 2020.

\bibitem{BLP17}
Andrea Bonito, Wenyu Lei, and Joseph~E. Pasciak.
\newblock On sinc quadrature approximations of fractional powers of regularly
  accretive operators.
\newblock {\em J. Numer. Math.}, 27(2):57--68, 2019.

\bibitem{bonito2012convergence}
Andrea Bonito and Joseph Pasciak.
\newblock Convergence analysis of variational and non-variational multigrid
  algorithms for the laplace-beltrami operator.
\newblock {\em Mathematics of Computation}, 81(279):1263--1288, 2012.

\bibitem{BP15}
Andrea Bonito and Joseph~E. Pasciak.
\newblock Numerical approximation of fractional powers of regularly accretive
  operators.
\newblock {\em IMA J. Numer. Anal.}, 37(3):1245--1273, 2017.

\bibitem{BX91}
James~H. Bramble and Jinchao Xu.
\newblock Some estimates for a weighted {$L^2$} projection.
\newblock {\em Math. Comp.}, 56(194):463--476, 1991.

\bibitem{DDEH09}
Klaus Deckelnick, Gerhard Dziuk, Charles~M. Elliott, and Claus-Justus Heine.
\newblock An {$h$}-narrow band finite-element method for elliptic equations on
  implicit surfaces.
\newblock {\em IMA J. Numer. Anal.}, 30(2):351--376, 2010.

\bibitem{DER14}
Klaus Deckelnick, Charles~M. Elliott, and Thomas Ranner.
\newblock Unfitted finite element methods using bulk meshes for surface partial
  differential equations.
\newblock {\em SIAM J. Numer. Anal.}, 52(4):2137--2162, 2014.

\bibitem{demlow09}
Alan Demlow.
\newblock Higher-order finite element methods and pointwise error estimates for
  elliptic problems on surfaces.
\newblock {\em SIAM J. Numer. Anal.}, 47(2):805--827, 2009.

\bibitem{demlow2007adaptive}
Alan Demlow and Gerhard Dziuk.
\newblock An adaptive finite element method for the laplace--beltrami operator
  on implicitly defined surfaces.
\newblock {\em SIAM Journal on Numerical Analysis}, 45(1):421--442, 2007.

\bibitem{dziuk88}
Gerhard Dziuk.
\newblock Finite elements for the {B}eltrami operator on arbitrary surfaces.
\newblock In {\em Partial differential equations and calculus of variations},
  volume 1357 of {\em Lecture Notes in Math.}, pages 142--155. Springer,
  Berlin, 1988.

\bibitem{Dziuk13}
Gerhard Dziuk and Charles~M. Elliott.
\newblock Finite element methods for surface {PDE}s.
\newblock {\em Acta Numer.}, 22:289--396, 2013.

\bibitem{guermond2009lbb}
J-L Guermond.
\newblock The lbb condition in fractional sobolev spaces and applications.
\newblock {\em IMA journal of numerical analysis}, 29(3):790--805, 2009.

\bibitem{jansson2021surface}
Erik Jansson, Mih{\'a}ly Kov{\'a}cs, and Annika Lang.
\newblock Surface finite element approximation of spherical
  whittle--mat$\backslash$'ern gaussian random fields.
\newblock {\em arXiv preprint arXiv:2102.08822}, 2021.

\bibitem{Kato61}
Tosio Kato.
\newblock Fractional powers of dissipative operators.
\newblock {\em J. Math. Soc. Japan}, 13:246--274, 1961.

\bibitem{kornhuber2008multigrid}
Ralf Kornhuber and Harry Yserentant.
\newblock Multigrid methods for discrete elliptic problems on triangular
  surfaces.
\newblock {\em Computing and Visualization in Science}, 11(4-6):251--257, 2008.

\bibitem{lang2015isotropic}
Annika Lang, Christoph Schwab, et~al.
\newblock Isotropic gaussian random fields on the sphere: regularity, fast
  simulation and stochastic partial differential equations.
\newblock {\em The Annals of Applied Probability}, 25(6):3047--3094, 2015.

\bibitem{lindgren2011explicit}
Finn Lindgren, H{\aa}vard Rue, and Johan Lindstr{\"o}m.
\newblock An explicit link between gaussian fields and gaussian markov random
  fields: the stochastic partial differential equation approach.
\newblock {\em Journal of the Royal Statistical Society: Series B (Statistical
  Methodology)}, 73(4):423--498, 2011.

\bibitem{lischke2018fractional}
Anna Lischke, Guofei Pang, Mamikon Gulian, Fangying Song, Christian Glusa,
  Xiaoning Zheng, Zhiping Mao, Wei Cai, Mark~M. Meerschaert, Mark Ainsworth,
  and George~Em Karniadakis.
\newblock What is the fractional laplacian? {A} comparative review with new
  results.
\newblock {\em Journal of Computational Physics}, 404:109009, 2020.

\bibitem{mekchay2011afem}
Khamron Mekchay, Pedro Morin, and Ricardo Nochetto.
\newblock Afem for the laplace-beltrami operator on graphs: design and
  conditional contraction property.
\newblock {\em Mathematics of computation}, 80(274):625--648, 2011.

\bibitem{ORG09}
Maxim~A. Olshanskii, Arnold Reusken, and J\"{o}rg Grande.
\newblock A finite element method for elliptic equations on surfaces.
\newblock {\em SIAM J. Numer. Anal.}, 47(5):3339--3358, 2009.

\bibitem{reusken14}
Arnold Reusken.
\newblock Analysis of trace finite element methods for surface partial
  differential equations.
\newblock {\em IMA J. Numer. Anal.}, 35(4):1568--1590, 2015.

\bibitem{scott1990finite}
L~Ridgway Scott and Shangyou Zhang.
\newblock Finite element interpolation of nonsmooth functions satisfying
  boundary conditions.
\newblock {\em Mathematics of Computation}, 54(190):483--493, 1990.

\end{thebibliography}

\end{document}